\documentclass[11pt]{amsart} 

\usepackage{fullpage}

\usepackage{url}

\usepackage{amssymb}

\usepackage{cite}

\usepackage{graphicx}

\usepackage[usenames]{color}

\usepackage{accents}

\renewcommand{\a}{\alpha}
\renewcommand{\b}{\beta}
\newcommand{\g}{\gamma}
\renewcommand{\d}{\delta}
\newcommand{\D}{\Delta}

\newcommand{\f}{\varphi}
\newcommand{\s}{\sigma}
\newcommand{\Si}{\Sigma}
\renewcommand{\k}{\kappa}
\renewcommand{\l}{\lambda}
\newcommand{\z}{\zeta} 
\renewcommand{\t}{\theta}
\renewcommand{\O}{\Omega}

\newcommand{\cF}{{\mathcal F}}
\newcommand{\cC}{{\mathcal C}}
\newcommand{\cM}{{\mathcal M}}

\newcommand{\cB}{{\mathcal B}}
\newcommand{\cL}{{\mathcal L}}
\newcommand{\cE}{{\mathcal E}}
\newcommand{\cU}{{\mathcal U}}
\newcommand{\cV}{{\mathcal V}}
\newcommand{\cP}{{\mathcal P}}
\newcommand{\cN}{{\mathcal N}}

\newcommand{\cH}{{\mathcal H}}

\newcommand{\bR}{\mathbb R}

\newcommand{\bZ}{\mathbb Z}

\newcommand{\bE}{\mathbb E}

\newcommand{\be}{\begin{equation}}
\newcommand{\ee}{\end{equation}}

\newcommand{\p}{\partial}

\newcommand{\tr}{\mathrm{tr}}

\newcommand{\beaa}{\begin{eqnarray*}}
\newcommand{\bea}{\begin{eqnarray}}
\newcommand{\beal}[1]{\begin{eqnarray}\label{#1}}
\newcommand{\bean}{\begin{eqnarray}\nonumber}
\newcommand{\beadl}[1]{\begin{deqarr}\label{#1}}
\newcommand{\eeadl}[1]{\arrlabel{#1}\end{deqarr}}
\newcommand{\eeal}[1]{\label{#1}\end{eqnarray}}
\newcommand{\eead}[1]{\end{deqarr}}
\newcommand{\eea}{\end{eqnarray}}
\newcommand{\eeaa}{\end{eqnarray*}}

\renewcommand{\to}{\rightarrow}

\DeclareMathOperator{\area}{area}

\renewcommand{\phi}{\varphi}
\renewcommand{\epsilon}{\varepsilon}
\renewcommand{\hat}{\widehat}

\newcommand{\<}{\langle}
\renewcommand{\>}{\rangle}

\newcommand{\dm}{{\partial M}}

\newcommand{\w}{\widetilde}

\theoremstyle{plain}
\newtheorem{theorem}{Theorem}[section]

\newtheorem{remark}[theorem]{Remark}

\newtheorem{lemma}[theorem]{Lemma}

\newtheorem{proposition}[theorem]{Proposition}

\theoremstyle{definition}

\def\blacksquare{\hbox to .60em {\vrule width .60em height .60em}}

\numberwithin{equation}{section}

\begin{document}

\title[ ]{The Bartnik quasi-local mass conjectures}

\author{Michael T. Anderson}
\address{Dept. of Mathematics, 
Stony Brook University,
Stony Brook, NY 11794}
\email{anderson@math.sunysb.edu}

\begin{abstract}
This paper is a tribute to Robert Bartnik and his work and conjectures on quasi-local mass. We present a framework in which to 
clearly analyse Bartnik's static vacuum extension conjecture. While we prove that this conjecture is not true in general,  
it remains a fundamental open problem to understand the realm of its validity. 
 
\end{abstract}

\maketitle 

\begin{flushright} In Memory of Robert Bartnik
\end{flushright}

\setcounter{section}{0}
\setcounter{equation}{0}

\section{Introduction}
\setcounter{equation}{0}

  This paper is a tribute to Robert Bartnik and his work on quasi-local mass in General Relativity. I met Robert in person only 
 twice, at the Newton Institute in 2005 and then a few years later in Oberwolfach and my personal interactions with him were sadly somewhat limited. 
 Robert's conjectures on the realization of the quasi-local mass, related to the structure of the space of static vacuum solutions of the Einstein equations 
 have had a profound impact on a significant portion of my research over the last 20 years. Personally, I was not trained early in GR and came to the 
 subject rather indirectly through a study of the geometrization problem for 3-manifolds (the Thurston conjecture). In that approach to geometrization, 
 a surprising connection arose between the degeneration of metrics of controlled scalar curvature and the structure of static vacuum Einstein metrics in 3 
dimensions.\footnote{The singularity models of such degeneration are static vacuum Einstein metrics, much like the singularity models of Ricci flow are 
 Ricci solitons. These two equations are superficially somewhat similar and it would be of interest to understand any deeper relations between them.} 
 I hope this paper may contribute in a small way to a deeper understanding and appreciation of the Bartnik quasi-local mass conjectures. 
 
 \medskip 
 
  The Bartnik mass of a bounded domain $\O$ with smooth boundary $\partial \Omega$ is a very natural and direct localization of the global 
ADM mass $m_{ADM}$. The original definition of Bartnik is the following, cf.~\cite{Ba1}, \cite{Ba2}. Let $\O$ be a 3-dimensional domain with smooth 
connected boundary $\partial \O$ (we will usually assume $\partial \O \simeq S^2$) and let $g_{\O}$ be a smooth Riemannian metric on 
$\O$, smooth up to $\partial \O$, with non-negative scalar curvature, $s_{g_{\O}} \geq 0$. Let $(M, g)$ be an asymptotically flat (AF) extension 
of $(\O, g_{\O})$; thus $\dm = \partial \O$ and the union $\O \cup M$ is a smooth, complete AF Riemannian manifold with non-negative and 
integrable scalar curvature. Assume in addition that $(M, g)$ has no horizons, i.e.~$(M, g)$ has no minimal surfaces surrounding $\dm$. 
Let $\cP_{\O}$ denote the set of such AF extensions $(M, g)$. The Bartnik mass of the smooth bounded domain $(\O, g_{\O})$ is then defined by 
\be \label{bm}
m_{B}(\O, g_{\O}) = \inf  \{m_{ADM}(M, g): (M, g) \in \cP_{\O}\},
\ee
where the infimum is taken over all $(M, g) \in \cP_{\O}$. 

    Bartnik observed in \cite{Ba1}, \cite{Ba2} that an AF extension $(M, g)$ of $\partial \O$ with $\hat M = \Omega \cup M$ which 
realizes the infimum in \eqref{bm} will in general only be Lipschitz along the ``seam'' $\partial \Omega = \dm$. By a 
simple and elegant argument using the $2^{\rm nd}$ variational formula for area, he showed that a minimizer should obey 
the boundary conditions 
\be \label{bcont}
\g_{\partial \O} = \g_{\dm}, \ \ \ \ H_{\partial \O} = H_{\dm},
\ee
where $\g$ is the induced metric on the boundary and $H$ is the mean curvature of $\dm$ with respect to the unit normal pointing into $M$, 
both with respect to $g_{\O}$ and $g$ respectively. The relation \eqref{bcont} 
implies that the scalar curvature is defined as a non-negative distribution across the seam $\dm = \partial \O$.\footnote{Actually this holds for 
$H_{\partial \O} \geq H_{\dm}$ and this condition is sometimes used instead of \eqref{bcont}, cf.~\cite{J}.} Standard minimal surface 
arguments show that if $H_{\dm} \leq 0$ then any extension $(M, g)$ has a horizon, so that it is common practice to assume 
\be \label{H}
H = H_{\dm} > 0.
\ee
As discussed in more detail in \S 2, there are several further reasons why the restriction \eqref{H} is essential. In the following, we will 
always assume \eqref{H}, unless explicitly noted otherwise. 

  The quasi-local data $(\g, H)$ of $\O$ on $\p \O = \dm$ are now called Bartnik boundary data. There are a number of modifications or 
variations of the definition of the Bartnik mass $m_{B}$; we refer to \cite{J} and further references therein for a careful and detailed 
discussion.  

\medskip 

  By general principles, it is to be expected that an extension $(M, g)$ realizing the infimum in \eqref{bm} satisfies strong 
conditions. In \cite{Ba1}, \cite{Ba2}, Bartnik presented a natural physical argument that extensions $(M, g)$ realizing the infimum in 
\eqref{bm} should be solutions of the static vacuum Einstein equations. Briefly, any dynamical gravitational 
field carries energy and so mass  and so an extension of minimal mass should 
have no gravitational dynamics, i.e.~be time-independent. For similar reasons, a minimal-mass extension should have no mass coming 
from matter sources, and so be vacuum. A time-independent vacuum solution which is time-symmetric ($K = 0$) is static vacuum. 

   The static vacuum Einstein equations are the equations for a pair $(g, u)$ on $M$ where $u$ is a potential function (forming 
the lapse function of the space-time $\cM = I\times M$) given by 
\be \label{stat}
u{\rm Ric}_g = D^{2}u, \ \ \D u = 0,
\ee
where ${\rm Ric}_g$ is the Ricci curvature, $D^2 u$ is the Hessian of $u$ and $\D u$ is the Laplacian of $u$, all with respect to $g$. 
These equations are equivalent to the statement that the space-time metric 
\be \label{st}
g_{\cM} = -u^{2}dt^{2} + g
\ee
is Ricci-flat, i.e.~a vacuum solution of the Einstein equations.\footnote{This holds in both Lorentizian and Riemannian signature.} 
It is important both mathematically and physically to add the requirements that $u > 0$ in $M$ and $u \to 1$ at infinity.
  
  A clearer approach as to why the static vacuum equations \eqref{stat} arise as minimizers was suggested by Bartnik in \cite{Ba3}, following 
a proposal by Brill-Deser-Fadeev \cite{BDF} regarding the positive mass theorem. Thus consider the Regge-Teitelboim Hamiltonian (with 
zero shift) given by 
\be \label{RT}
\cH_{RT} = -\int_M us_g dv_g +16\pi m_{ADM},
\ee
where $s_g = {\rm tr}_g {\rm Ric}_g$ is the scalar curvature. 
The first term corresponds to the Einstein-Hilbert action while the mass term is introduced to give a well-defined, and in particular 
differentiable variational problem on the space of fields. The configuration space of fields is here given by AF static metrics as in 
\eqref{st}, but not necessarily vacuum, i.e.~pairs $(g, u)$ with $u > 0$; the AF conditions are described more precisely in \S 2. The 
Hamiltonian $\cH_{RT}$ is thus a smooth function on such static pairs $(g, u)$. 

Let 
$$\cC = \{g: R_g = 0\}.$$
This is the vacuum constraint set (in the time-symmetric case $K = 0$). Let $\cC_{(\g,H)} \subset \cC$ be the subset where the 
boundary data $(\g, H)$ induced by $g$ is fixed. A necessary first step in the Bartnik program is to prove that $\cC_{(\g,H)}$ is a smooth (Banach) 
manifold; this was proved to be the case in \cite{AJ}. The Hamiltonian $\cH_{RT}$ then 
restricts to be the smooth function 
$$\cH_{RT} = 16\pi m_{ADM}: \cC_{(\g,H)} \to \bR$$
on $\cC_{(\g,H)}$. Thus, critical points of the ADM mass on $\cC_{(\g,H)}$ are just the critical points of $\cH_{RT}$ on $\cC_{(\g,H)}$. It is then 
straightforward to verify that such critical points are exactly static vacuum Einstein metrics. Further,  cf. ~\cite{AJ}, one has $u > 0$ 
everywhere and $u \to 1$ at infinity, at least for critical points realizing the infimum of $m_{ADM}$ on $\cC_{(\g,H)}$, i.e.~the Bartnik mass. 
 
   There is a useful analogy, cf.~\cite{Ba2}, \cite{Ba4}, \cite{HI}, of the Bartnik mass with the gravitational capacity of a body $\O \subset \bR^{3}$ 
in Newtonian gravity, or a charged body in electrostatics. Here one minimizes the Dirichlet energy,
\be \label{Eu}
E(v) = \int_{M}|dv|^{2},
\ee
over $M = \bR^{3}\setminus \O$ with boundary conditions $v = 1$ at $\partial \O$ and $v \to 0$ at infinity.\footnote{In \eqref{Eu} and below 
the volume form will often be dropped from the notation when its choice is obvious.} (One could also 
set $v' = 1-v$ with $v' \to 1$ at infinity). Classical results show that the infimum of \eqref{Eu} is realized by a unique harmonic 
function $v_0$, $\D v_0 = 0$ on $M$; $v_0$ represents the gravitational potential of the single layer $\partial \Omega$. The capacity 
of $\O$, equal to the total mass or charge up to a constant, is then given by 
\be \label{Nu}
E(\O) = \inf E(v) = -\int_{\dm}N(v_0),
\ee
where $N$ is the unit normal into $M$ at $\dm$. 

   Partly based on the Newtonian analogy, Bartnik \cite{Ba1}, \cite{Ba2}, made the bold conjecture that a minimizer $(M, g)$ in \eqref{bm} 
also always exists and is unique.

\medskip 

{\bf Conjecture I (Bartnik Minimization Conjecture).}  For any given smooth boundary data $(\g, H)$, $H > 0$, of a domain $\O$, $\p \O \simeq S^2$, 
the infimum in \eqref{bm} is realized by a unique (up to isometry) AF extension $(M, g)$, i.e. 
$$m_{ADM}(M, g) = m_B.$$ 
Moreover, there is a positive potential function $u > 0$ on $M \simeq \bR^3 \setminus B$ with $u \to 1$ at infinity such that the triple  $(M, g, u)$ 
is a solution of the static vacuum Einstein equations \eqref{stat} with boundary data $(\g,H)$.  

\medskip 

  The validity of this conjecture would ensure that the Bartnik mass is well-behaved as a function of the boundary data $(\g, H)$ coming from 
the solid body $(\O, g_{\O})$. For example, it would follow that the Bartnik mass is a smooth function of the boundary data $(\g, H)$; in fact 
the mass would be effectively computable via the expression 
$$m_{B} = \frac{1}{4\pi}\int_{\dm} N(u),$$
where $(M, g, u)$ is the unique static vacuum extension of $(\g, H)$. 

If true, Conjecture I suggests the following conjecture \cite{Ba1}, \cite{Ba2}, which thus serves as a test 
of the minimization conjecture but is also of independent interest in geometric PDE theory. 

\medskip 

{\bf Conjecture II (Bartnik Static Extension Conjecture).}  Given smooth boundary data $(\g, H)$, $H > 0$, on $\dm \simeq S^2$, there 
exists a unique (up to isometry) AF solution $(M, g, u)$, $M \simeq \bR^3 \setminus B$, $u > 0$, $u \to 1$ at infinity, of the static 
vacuum Einstein equations \eqref{stat} which induces the data $(\g, H)$ at $\dm$. 

\medskip 

  Conjecture II may be considered as a pure PDE problem: it concerns the unique global solvability of a non-linear elliptic boundary 
 value problem. Conjecture I does not logically imply Conjecture II since a static vacuum solution may not minimize the Bartnik mass. 
 Nevertheless, we expect that any counterexample to Conjecture I leads also to counterexamples to Conjecture II. Of course 
 counterexamples to Conjecture II are counterexamples to Conjecture I. 

\medskip 

   The static vacuum Einstein equations are seemingly among the simplest set of geometric (diffeomorphism invariant) PDE for a general 
metric $g$, either on a Lorentzian or Riemannian 4-manifold. In the Lorentzian or General Relativity setting, the simplest vacuum 
Einstein solutions are static, i.e.~time independent with vanishing initial momentum and so have a hypersurface-orthogonal time-like 
Killing field.\footnote{The focus here is on static case; for recent progress on the more difficult stationary case, cf.~\cite{An}.}  
Similarly, Einstein metrics on 4-manifolds are of fundametal interest in Riemannian geometry and the existence of a 
global non-vanishing  Killing field reduces the Einstein PDE system to the simpler system \eqref{stat} on a 3-manifold $M$. 
While the equations \eqref{stat} lie at the crossroads or intersection of Lorentzian and Riemannian geometry, gaining an understanding 
of the space of solutions and of Conjectures I and II has remained very difficult. One could impose a further symmetry, i.e.~Killing field 
to simplify the problem further to 2 dimensions. This leads to the class of axi-symmetric or Weyl solutions, discussed further in 
detail in \S 4: however, even in this case, Conjectures I and II remain very challenging. 
  
    For Conjecture II, the source of these difficulties is of course the non-linear nature of the PDE \eqref{stat} and the boundary data $(\g, H)$. 
This is also a coupled system of PDE, adding significantly to the complexity. To gain some general perspective, it is worth comparing 
the global existence and uniqueness problem for static vacuum Einstein metrics with other natural geometric problems. In fact, 
global existence and uniqueness for solutions of nonlinear elliptic geometric PDE is rather rare and is much more the exception 
than the rule. The list of examples is huge, so we mention only the following illustrative examples:

\medskip 
  Global existence and uniqueness:
  
$\bullet$ K\"ahler-Einstein metrics. Among the few situations where global existence and uniqueness hold, the most celebrated is Yau's solution 
of the Calabi problem for compact K\"ahler manifolds with $c_1 = 0$ and the Aubin-Yau theorem in the case $c_1 < 0$, as well as the more 
recent work of Chen-Donaldson-Sun in the more difficult positive case $c_1 > 0$. These existence proofs use a version of the continuity 
method, discussed briefly in \S 2. 

$\bullet$ Weyl embedding problem. Any smooth metric of positive Gauss curvature on $S^2$ admits a smooth isometric embedding 
into $\bR^3$ as the boundary of a convex body, unique up to rigid motion of $\bR^3$. Again, the existence proof of this result by Nirenberg 
uses the continuity method. 

\medskip 

Global existence, non-uniqueness:

$\bullet$ The Yamabe problem. Global existence was established by variational methods through the combined work of Yamabe, Trudinger, 
Aubin and Schoen. However, it is well-known that uniqueness, either of constant scalar curvature metrics or of Yamabe minimizers in a 
conformal class, fails in general. 

$\bullet$ Minimal varieties. There is a large and very satisfactory theory for existence of minimal varieties spanning a given boundary in an 
ambient Riemannian manifold, using either parametric methods in low dimensions (Plateau problem) or non-parametric methods (geometric 
measure theory) in all dimensions. However, it is again well-known that uniqueness fails in many situations. 

\medskip 

Failure of global existence and uniqueness: 

$\bullet$ Nirenberg problem. A very simple sounding problem is the Nirenberg problem of determining the existence and uniqueness of metrics on 
$S^2$ conformal to the round metric with prescribed Gauss curvature. Here, both existence and uniqueness fail in general and the number of 
solutions depends in a very complicated way on the structure of the Gauss curvature function $K: S^2 \to \bR$. 

$\bullet$ Vacuum Einstein constraint equations in GR. The main method of solving these constraint equations is the conformal method of 
Lichnerowicz, Choquet-Bruhat and York. Here one has full existence and uniqueness results for constant mean curvature (CMC) data, but 
failure of both existence and uniqueness for data far from CMC data. 

One could easily go on and on; but the lesson is that non-linear geometric problems usually have complicated global behavior. 

\medskip 

We note that there is now also a large and growing literature of estimates of the Bartnik mass $m_B(\g, H)$ from below and above in terms of the 
boundary data $(\g, H)$, as well as comparisons of $m_B$ with other notions of quasi-local mass, such as the Hawking, Brown-York, or Wang-Yau masses. 
These estimates and comparisons may be viewed as analogs of the extensive and well-known estimates of the capacity of a domain $\O \subset \bR^3$, 
or of the first eigenvalue $\l_1$ of a compact Riemannian manifold, in terms of other geometric data. 

However, the theory behind the capacity of a domain or of the first eigenvalue of the Laplacian is simple and easy to establish; they are realized by 
a harmonic function or a first eigenfunction whose existence and uniqueness is easy to establish, since the Laplacian is a single scalar, linear PDE. 
The focus is this paper is instead on the theory underlying the definition of Bartnik mass, in particular on the foundational Conjectures I and II. 
The Bartnik mass formally exists as a number, but one should understand its structure and global behavior as a function of the boundary data 
$(\g, H)$. 

\medskip 

  In accord with most of the examples discussed above, it turns out that both Conjectures I and II are false in full generality. 
For example, we prove: 

\begin{theorem} 
There is no static vacuum Einstein metric $(M, g, u)$ smooth up to $\dm$ with boundary data $(\g, H)$, where $\g$ is any 
smooth metric on $S^2$ for which the Gauss curvature $K_{\g}$ is a Morse function on $S^2$ and 
$$0 < H < H_0,$$
where $H_0$ is a positive constant, sufficiently small depending on $\g$. 
\end{theorem} 

We refer to Theorem 3.5 below for the exact statement. We also discuss in \S 3 why this non-existence result also leads to 
expected local non-uniqueness for Conjecture II. Conjecture I was previously proved to be false in certain (quite different) 
situations in \cite{AJ}. 

   It remains then a fundamental challenge to identify conditions or restrictions on the boundary data $(\g, H)$ which 
ensure the validity of the Bartnik Conjectures. 

\medskip 

  A brief summary of the contents of this paper is as follows. In \S 2, we discuss the basic framework in which to analyse Conjecture II 
and show that the local structures involved are very well behaved. In \S 3, we turn to global issues of existence and uniqueness and prove 
in particular Theorem 1.1. A large class of relatively explicit examples of static vacuum Einstein metrics are the Weyl metrics, discussed 
in detail in \S 4. Finally, in an Appendix, \S 5, we prove a regularity result for (local) minimal surfaces in static vacuum Einstein metrics needed 
for the proof of Theorem 3.5. 

\section{Local Behavior}

Let $\cB$ be the set of Bartnik boundary data $(\g, H)$ on $S^2$ with $H > 0$ as in \eqref{H}. Let $\cE$ be the set of isometry classes of 
static vacuum Einstein solutions with $H > 0$ at the boundary $\dm$.\footnote{It is an open question whether $\cE$ is connected or not. Thus 
$\cE$ is understood to be the component containing the standard flat exterior metric $\bR^3 \setminus B^3(1)$, with $g$ the Euclidean 
metric and $u \equiv 1$.} Thus $\cE$ is the set of equivalence classes of AF static vacuum 
Einstein metrics $(M, g, u)$, modulo the equivalence relation 
$$(g_1, u_1) \sim (g_2, u_2)$$
if there is a diffeomorphism $\psi:M \to M$ with $\psi = {\rm Id}$ on $\dm$ and $\psi \to {\rm Id}$ at infinity, (cf.~\eqref{AF3} below), such that 
$\psi^* g_2 = \psi^*g_1$ and $\psi^* u_2 = u_1$. One has a natural 
Bartnik boundary map 
\be \label{bmap}
\Pi_B: \cE \to \cB,
\ee
$$ (g, u) \to (\g, H).$$ 
Thus the static extension Conjecture II is equivalent to the statement that $\Pi_B$ is a bijection. This means that the space of solutions $(M, g, u)$ (up to 
isometry) should be set-theoretically parametrized by its boundary data $(\g, H)$. However, it is natural to introduce topology at this point. 
Well-posedness in PDE also typically assumes stability properties; small perturbations of the data (boundary data in this situation, initial or 
initial$+$boundary data in the case of parabolic or hyperbolic PDE) implies small perturbations of the solution; this corresponds to 
continuity of the inverse map $\Pi_B^{-1}$ in \eqref{bmap}. 

   The set $\cB$ of boundary data may be naturally topologized in several ways. For convenience, we choose here the H\"older space topology, 
 $$\cB = {\rm Met}^{m,\a}(S^2) \times C_+^{m-1,\a}(S^2).$$ 
 Here ${\rm Met}^{m,\a}(S^2)$ is the space of $C^{m,\a}$ Riemannian metrics on $S^2$ and $C_+^{m-1,\a}(S^2)$ is the space of positive $C^{m-1,\a}$ 
 functions on $S^2$. The space $\cB$ has the structure of a Banach manifold (an open subset of a linear Banach space)\footnote{In place of 
 $C^{m,a}$, it is actually better to choose the so-called little H\"older spaces $c^{m,\a}$, i.e. the closure of $C^{\infty}$ in the $C^{m,\a}$ H\"older norm, 
 which has the structure of a separable Banach space; however, this is not an essential distinction in the discussion to follow.}. 
 
 Next, let ${\rm Met}_{\d}^{m,\a}(M)$ be the space of $C^{m,\a}$ Riemannian metrics $g$ on $M \simeq \bR^3 \setminus B$ which are 
 asymptotically flat of order $\d$, where $\d \in (\frac{1}{2}, 1)$, i.e.~for $r$ sufficiently large, 
 \be \label{AF1}
 |g - g_{Eucl}| \leq cr^{-\d}, \ \ |\partial^{k+\a}g| \leq cr^{-\d-k-\a},
 \ee
$1 \leq k \leq m$. 
 Similarly, let $C_+^{m,\a}(M)$ denote the space of positive potential functions $u$ on $M$ which are asymptotically flat of order $\d$:
 \be \label{AF2}
 |u - 1| \leq cr^{-\d}, \ \ |\partial^{k+\a} u| \leq cr^{-\d-k-\a}.
 \ee
 The pair $(g, u)$ assemble to form the static or time-independent metric on the 4-manifold $\cM$, as in \eqref{st}. 
A $C^{m+1,\a}$ diffeomorphism $\psi: M \to M$ is AF if  
 \be \label{AF3}
 |\psi - {\rm Id}| \leq cr^{-\d}, \ \ |\partial^{k+\a}\psi| \leq cr^{-\d-k-\a},
 \ee
for $1 \leq k \leq m+1$. 
 
 The product ${\rm Met}_{\d}^{m,\a}(M)\times C_+^{m,\a}(M)$ then also has the structure of a Banach manifold. Let $\bE \subset 
 {\rm Met}_{\d}^{m,\a}(M)\times C_+^{m,\a}(M)$ be the subset of solutions of the static vacuum Einstein equations \eqref{stat} with 
 $H > 0$ at $\dm$. It is proved in \cite{AK} that $\bE$ is a smooth Banach submanifold of ${\rm Met}_{\d}^{m,\a}(M)\times C_+^{m,\a}(M)$. 
 Moreover, the group ${\rm Diff}_1^{m+1,\a}(M)$ of $C^{m+1,\a}$ AF diffeomorphisms of $M$ equal to the identity on the boundary 
 $\dm$ and asymptotic to ${\rm Id}$ at infinity as in \eqref{AF3}, acts smoothly, freely and properly on $\bE$, with a smooth local slice. 
 Consequently, the quotient 
 \be \label{cE}
\cE = \bE / {\rm Diff}_1^{m+1,\a}(M),
\ee
representing isometry classes of solutions, is also a smooth Banach manifold. 

  In addition, the Bartnik boundary map $\Pi_B$ in \eqref{bmap} is a smooth Fredholm map, of Fredholm index 0, cf.~again \cite{AK}. 
Thus the linearization or derivative $D_{(g,u)}\Pi_B$ at any point $(M,g,u)$ is a Fredholm map of Banach spaces. This 
means that the linear map $D\Pi_B$ has closed range and finite dimensional kernel and cokernel. In more detail, let 
$K$ be the finite dimensional kernel of $D\Pi_B$ at some point $(M, g, u) \in \cE$. The closed subspace $K$ admits a closed 
complement $D$, 
$$T_{(g,u)}\cE = D \oplus K,$$
and similarly the image, $W = {\rm Im} D\Pi_B$ is a closed complemented subspace of $T_{(\g,H)}\cB$. Let $C$ denote the finite dimensional 
complement. Then $dim C = dim K = k < \infty$ since $D\Pi_B$ has Fredholm index $0$. The restricted map 
$$D_{(g,u)}\Pi |_D: D \to W,$$
is an isomorphism. 

If $dim K = 0$, then $D\Pi_B$ itself is an isomorphism and so the inverse function theorem for Banach manifolds implies $\Pi_B$ is a 
local diffeomorphism near $(M,g,u)$. Suppose $dim K = k > 0$. Then the implicit function theorem (also known as the constant rank theorem) 
for Banach manifolds states that there is a local submanifold $\cU$ of codimension $k$ in $\cE$, with $T_{(g,u)}\cU \oplus K = T_{(g,u)}\cE$ 
and a local submanifold $\cV$ of codimension $k$ in the boundary data space $\cB$, with $T_{(\g,H)}\cV = W$ such that 
\be \label{UV}
\Pi_B | _{\cU}: \cU \to \cV
\ee 
is a diffeomorphism. Thus, near any given solution $(M, g, u)$, one has local existence (and local uniqueness) of solutions, parametrized 
by boundary data $(\g, H)$ in an open set of a local submanifold of codimension $k$ in $\cB$. 

 We see then that the local structure of the space of solutions $\cE$ and the local behavior of the map $\Pi_B$ relating solutions with boundary 
 data are both very well behaved. However, this structure of the mapping $\Pi_B$ breaks down at the locus $H = 0$ (minimal surface boundary) 
 basically due to the black hole uniqueness theorem. To see this, let $\bar \cE$ be the closure of $\cE$ in the $C^{m,\a}$ norm. This consists of 
 static vacuum Einstein metrics $(M, g, u)$ with $u \geq 0$ on $\dm$.\footnote{It does not really make sense in this framework to consider metrics 
 where ``$u < 0$" somewhere.} Similarly, let $\w \cB$ be the full space of boundary data, without the restriction that $H > 0$. One still has the 
 Bartnik boundary map 
\be \label{PiB2}
 \Pi_B: \bar \cE \to \w \cB,
 \ee
 $$(M, g, u) \to (\g, H).$$
  
  \begin{proposition} 
 Let $(M, g, u)$ be an AF static vacuum Einstein metric in $\bar \cE$ and let $\Si$ be a minimal surface surrounding $\dm$, 
 (possibly $\Si = \dm$). Then $\Si = \dm$ and $(M, g, u)$ is a Schwarzschild metric with $u = 0$ on the horizon $\dm$. 
 \end{proposition} 
 
 \begin{proof} This is basically a consequence of the well-known black hole uniqueness theorem \cite{I}, \cite{R}, \cite{BM} and its extension by an 
 elegant argument of Galloway \cite{G}. We refer also to \cite{Mi} for a different proof, (which relies however on the more difficult positive mass 
 theorem). 
 
   To begin, without loss of generality, by standard minimal surface arguments we may assume $\Si$ is outer-minimizing, cf.~\eqref{om} below, 
and so in particular $\Si$ is a stable minimal surface. By the $2^{\rm nd}$ variational formula for area 
\be \label{2nd}
\int_{\Si}|df|^2 - (|A|^2 + {\rm Ric}(N,N))f^2 \geq 0,
\ee
for all Lipschitz functions $f$ with $f \geq 0$; clearly \eqref{2nd} then also holds for all Lipschitz functions $f$. By the static vacuum equations 
\eqref{stat}, we have  
$$-{\rm Ric}(N,N) = -u^{-1}NN(u) = u^{-1}\D_{\Si}u ,$$
since $H = 0$, cf.~also the line preceding \eqref{Hamcon2}. Choosing $f = u$ in \eqref{2nd} gives
$$\int_{\Si} |du|^2 + u\D_{\Si}u - u^2|A|^2 \geq 0,$$
and so by the divergence theorem 
$$\int_{\Si}u^2|A|^2 = 0.$$
Again, the static vacuum Einstein equations imply that $A = 0$ at any locus where $u = 0$. Hence, it follows that 
\be \label{A0}
A = 0,
\ee
on $\Si$. If also $u = 0$ on $\Si$, then the result follows by the black hole uniqueness theorem, so assume $u$ is not identically $0$ on $\Si$. 
Since by the constraint equation \eqref{Hamcon}, $-2{\rm Ric}(N,N) = s_{\g_{\Si}}$, \eqref{2nd} becomes 
$$\int_{\Si} |df|^2 + K_{\Si}f^2 \geq 0.$$
Thus $u$ is a lowest eigenfunction of the operator $-\D + K_{\Si}$ and hence, (since $u$ is not identically $0$), $u > 0$ on $\Si$. 
 
   Now consider the (so-called) optical metric $\hat g = u^{-2}g$ on $(M, g)$ exterior to $\Si$ and the normal exponential map $\hat{exp}_{\Si}$ from $\Si$ 
in the metric $\hat g$. It is proved in \cite{G} that one has the monotonicity formula 
\be \label{mono}
\frac{d}{dt}(\frac{H}{u}) = -|A|^2,
\ee
where $d/dt$ is the $\hat g$-unit tangent to the flow lines of $\hat{exp}_{\Si}$; all other metric quantities in \eqref{mono} are with respect to $g$. 
(The sign for $H$ used here is opposite to that in \cite{G}). By \eqref{A0}, $\frac{d}{dt}(\frac{H}{u}) = 0$ at $t = 0$ and similarly 
$\frac{d^2}{dt^2}(\frac{H}{u}) = 0$ at $t = 0$. Thus 
\be \label{4}
\frac{d^3}{dt^3}(\frac{H}{u}) = -2|\frac{d}{dt}A|^2.
\ee
This corresponds to the $4^{\rm th}$ variational formula for area. By the outer-minimizing property the left side of \eqref{4} must be non-negative when 
integrated over $\Si$. Hence, it follows that 
$$\frac{d}{dt}A = 0,$$
at $t = 0$. By a standard computation, using the fact that $A = 0$,
$$\frac{d}{dt}A = -D^2u - R_N u = 0,$$
where $R_N(X,Y) = \<R(N,X)Y, N\>$. In 3 dimensions, one has $-R_N = {\rm Ric} + {\rm Ric}(N,N)g$, and so by the static vacuum equations 
\eqref{stat}, it follows that ${\rm Ric}(N,N) = 0$. In turn, by the Hamiltonian constraint \eqref{Hamcon}, this gives $K_{\Si} = 0$, which implies 
$\D_{\Si} u = 0$, cf.~\eqref{Hamcon2}. Hence $u = const$ on $\Si$. Also $N(u) = const$ on $\Si$ by the divergence constraint \eqref{dnu}. 
Thus the Cauchy data for $(M, g, u)$ at $\Si$ are the trivial data $(\g = flat, u = const, A = 0, N(u) = const)$. It then follows from unique continuation, 
or more simply just analyticity (static vacuum solutions with minimal boundary and $u > 0$ are analytic up to the boundary), that $(M, g)$ is flat 
and $u$ is an affine function. This again gives a contradiction; there are no compact minimal surfaces in $\bR^3$. 

 \end{proof} 
 
 \begin{remark} 
{\rm Proposition 2.1 is stable in the following sense: if $(M, g, u)$ is a static vacuum Einstein metric with boundary data $(\g, H)$ and 
$H$ is sufficiently small, and if the $C^{m,\a}$ norm, $m \geq 4$, of $(g, u)$ up to $\dm$ is uniformly bounded, then $(M, g, u)$ is close to a 
Schwarzschild metric $g_{Sch}(m)$ in $\bar \cE$. The proof of this is the same as that of Proposition 2.1, by taking a sequence of such metrics 
with boundary data $(\g_i, H_i) \to (\g, 0)$. 

   This statement will be significantly generalized in \S 3, cf.~Theorem 3.5. 
} 
\end{remark}

 Proposition 2.1 implies 
\be \label{Schdeg}
(\Pi_B)^{-1}({\rm Met}^{m,\a}(S^2), 0) = g_{Sch},
\ee
i.e.~ the inverse image of infinite dimensional space ${\rm Met}^{m,\a}(S^2) \times \{0\} \subset \w \cB$ is the 1-dimensional curve of 
Schwarzschild metrics. The linearized version of Proposition 2.1 implies that the image of $D_{g_{Sch}}\Pi_B$ has infinite codimension 
in $T\w \cB$ and so is not Fredholm. The Schwarzschild metrics lie at the boundary or edge of the space $\cE$ but the relation \eqref{Schdeg} 
shows this boundary is just a curve, not a hypersurface as one might expect. This indicates that the local behavior of  $\Pi_B$ must 
change significantly on approach to the locus $H = 0$ away from round metrics; this will be explored and discussed in more detail in \S 3. 
 
    Note also that the static vacuum equations \eqref{stat} are no longer strictly elliptic (in any gauge) at regions in the locus $\{u = 0\}$; instead 
the equations become degenerate elliptic. Namely, the leading order symbol for \eqref{stat} is $u\D g_{ij} = -2\p_i \p_j u$, $\D u = 0$ in a harmonic 
gauge for $g$, which is strictly elliptic only when $u > 0$. We note also that when $H > 0$ at the boundary $\dm$, a static vacuum solution 
$(M, g, u)$, $C^2$ smooth up to $\dm$ cannot have $u = 0$ at any point of $M$ or $\dm$. Namely, by \eqref{Hamcon2} 
 $$\D u + HN(u) = \frac{u}{2}(|A|^2 - H^2 + s_{\g}).$$
 If $u \geq 0$ on $\dm$ and $u(p) = 0$ at $p \in \dm$, then at $p$,  $\D u \geq 0$ and $N(u) > 0$ by the Hopf maximum principle. However, the 
 right side vanishes at $p$ giving a contradiction if $H(p) > 0$. 
    
      In sum, the map $\Pi_B$ is Fredholm only where $u > 0$ up to the boundary, and the requirement $H > 0$ at $\dm$ ensures this condition. 
All of the above are reasons for enforcing \eqref{H}. 
  
   \medskip 
 
    Points $(g, u) \in \cE$ where $D\Pi_B$ is an isomorphism are called regular or non-degenerate points of the map $\Pi_B$. Recall that boundary data 
$(\g, H) \in \cB$ is a regular value of $\Pi_B$ if every point in the inverse image $(\Pi_B)^{-1}(\g, H)$ is a regular point of $\Pi_B$ in $\cE$. By the 
Smale-Sard theorem \cite{Sm}, the regular values are generic (of second Baire category) in the target $\cB$. (Note that, by fiat, the empty set is a 
regular value). A point $(g, u)$ is singular (for $\Pi_B$) if it is not a regular point. 

  An important locus of singularities of a smooth index 0 Fredholm map $\cF: \cM \to \cN$ between Banach manifolds is a fold locus. This is given 
 by a (local) hypersurface $\cH \subset \cM$ (a codimension 1 submanifold) where $\cH$ consists of singular points of $\cF$ with 
 $dim {\rm Ker} \, D\cF = 1$ and the map $\cF$ is 2-1 in an open neighborhood of $\cH$ in $\cM$ off of $\cH$; compare with \eqref{UV}, \eqref{S1} 
 and the discussion following \eqref{eor} below. In the context of (Riemannian) Einstein metrics, the most well-known example of such fold 
behavior is the Hawking-Page phase transition for the AdS Schwarzschild metrics \cite{HP}. Similar fold behavior holds for the the interior 
Dirichlet boundary map for Ricci-flat Schwarzschild metrics \cite{AG}. Such behavior also occurs for the general interior Bartnik boundary 
map \cite{A5}, i.e.~for the analog of \eqref{bmap} where $M$ is replaced by a compact (interior) manifold with boundary. 

  It is natural to conjecture that generically, i.e.~at least on an open dense subset of $\cE$, the linearization $D\Pi_B$ is an isomorphism, 
i.e.~the regular or non-degenerate points in the domain $\cE$ are generic. This issue has been widely considered and analysed in many 
other geometric PDE (and ODE) problems. For example, for geodesics in Riemannian manifolds, this is the statement that the regular points 
of the exponential map $exp_p$ are open and dense in $T_pM$; generically, two points in $M$ are not conjugate along a geodesic, i.e.~there 
are no Jacobi fields vanishing at the end points. Similarly for minimal submanifolds, for a generic boundary in fixed Riemannian manifold, 
(or for a generic Riemannian metric), minimal surfaces have no Jacobi fields vanishing on the boundary. 

   In fact, based on an analogy with the monotonicity of the first eigenvalue of the Laplacian on bounded domains with Dirichlet boundary 
condition with respect to inclusion of domains, White \cite{Wh} proved that the set of minimal immersions with non-trivial Jacobi fields 
vanishing on the boundary is of codimension 1 in the space of all minimal immersions $(M, \dm) \to (N, g_N)$. This is of course a 
much stronger conclusion than generic. The proof in \cite{Wh} uses a variational approach, together with a normal interior variation of 
the domain of the minimal surface and the Calderon unique continuation theorem for elliptic systems. A similar result was proved for the 
interior Bartnik problem in \cite{A5} using an analogous approach, but  based on a unique continuation property for Einstein metrics and 
their linearization; this is considerably more difficult partly due to gauge issues.  
 
 Using essentially this same approach but with some modifications, An-Huang \cite{AH1}, \cite{AH2} have recently shown that the regular points 
 of $\Pi_B$ are indeed generic in many cases. 
 
\medskip 
 
   As discussed above, while the spaces $\cE$, $\cB$ and the boundary map $\Pi_B$ are well-behaved locally, this does not go very far towards 
resolving Conjecture II, the static vacuum extension conjecture. Consider for instance a very simple toy model. In place of $\cE$, consider the 
circle $S = S^1(r) \subset \bR^2$ of radius $r$ about the origin. Similarly, replace $\cB$ by the $x$-axis and the boundary $\Pi_B$ by the 
projection $\pi$ to the $x$-axis:
\be \label{S1}
\pi: S \to \bR, 
\ee
$$\pi(x,y) = x.$$
For $r$ small, the image of $\pi$ in $\bR$ is small; for most values in $\bR$ there is no solution. The space of boundary data realized by solutions 
increases as $r$ increases, but for any $r < \infty$, there is still a large space of boundary data without solutions. Only in the limit $r \to \infty$ where 
$S^1(r)$ becomes $\bR$ and $\pi$ becomes the identity (after suitable translations) does the situation improve. In all cases of course, the local 
structures involved are very well-behaved. There are only $2$ singular or degenerate points, both fold points, at $x = \pm r$; all other points in $S$ 
are regular.  

   In the example above, the degree of $\pi$ is $0$, so there are generically either 2 distinct solutions or no solutions. On the other hand, consider 
the following simple modification. Let $T$ be the graph of the function $\tan(x)$ over the interval $(-\frac{\pi}{2}, \frac{\pi}{2})$ and consider 
\be \label{T1}
\pi: T \to \bR,  
\ee
$$\pi(x,y) = x.$$
Then either there is a unique solution $y$ to $\pi(x, y) = x$ or there are no solutions. The degree of $\pi$ here is not well-defined. This is because 
the map $\pi$ is no longer proper; the issue of properness is fundamental in understanding the global behavior of non-linear mappings and will be 
discussed in more detail for the map $\Pi_B$ in $\S 3$. 

\medskip 
 
    One of the most effective methods of proving global existence and uniqueness is the continuity method, as mentioned above in the works on the 
solution to the Calabi conjecture and the solution to the Weyl embedding problem. In the current context, this method can be described as follows. 
Choose boundary data for which a static vacuum solution is known to exist; for instance $(\g_0, H_0) = (\g_{+1}, 2)$, which is realized by the exterior 
flat metric $M = \bR^3 \setminus B^3(1)$, $g = g_{Eucl}$ is the Euclidean metric and $u = 1$. Given any boundary data $(\g, H) \in \cB$, choose a 
path from $(\g_0, H_0)$ to $(\g, H)$ in $\cB$; for instance
$$(\g_s, H_s) = (1-s)(\g_0, H_0) + s(\g, H).$$
To prove there is a solution in $\cE$ with boundary data $(\g, H)$, one shows that the set of $s$ for which there is a solution is both open and 
closed in $[0,1]$. The openness property amounts to showing that solutions $(M, g_s, u_s)$ with $\Pi_B(M, g_s, u_s) = (\g_s, H_s)$ are always 
non-degenerate. In case the degeneracy locus is non-empty, for example a collection of fold hypersurfaces in $\cE$, then a very useful 
replacement or generalization of the method of continuity is the degree theory, e.g.~the Leray-Schauder degree or more generally 
the Smale degree \cite{Sm}. This does not require the non-degeneracy property, but to establish a well-defined degree does require 
the properness of the map $\Pi_B$, which is the analog of the closedness condition above.

\section{Global Behavior}

   The key issue in understanding the global behavior of $\Pi_B$ is to understand how well the boundary data $(\g, H)$ controls the 
behavior of any possible solution $(M, g, u)$ with such boundary data. As always with elliptic problems, this requires apriori estimates on the 
geometry of solutions $(M, g, u)$ in terms of $(\g, H)$. Proving such estimates is equivalent to proving that $\Pi_B$ is a proper map, 
i.e.~the inverse image $(\Pi_B)^{-1}(K)$ of compact sets $K \subset \cB$ are compact sets in the domain $\cE$. We point out that 
Fredholm maps are always locally proper, cf.~\cite{Sm}. 
 
   As a starting point of this analysis, we recall that one does have optimal apriori estimates of solutions $(M, g, u)$ in the interior of $M$, 
 away from $\dm$. Let $t(x) = dist_g(x, \dm)$ be the distance function to the boundary $\dm$ and let $\nu = \log u$. We recall that, by 
 definition, $u > 0$ in $M$. By \cite{A1}, there is a universal constant 
 $K < \infty$ such that 
\be \label{apri}
|{\rm Rm|}(x) \leq \frac{K}{t^2(x)}, \ \ |d \nu (x)| \leq \frac{K}{t(x)},
\ee
for any static vacuum solution $(M, g, u)$; here ${\rm Rm}$ is the Riemann curvature tensor of $g$. The estimates \eqref{apri} are scale invariant 
and such scale-invariant estimates also hold for all higher derivatives. It follows (from the Cheeger-Gromov compactness theorem) that if one 
has weak uniform control on the global geometry of the locus $L_1 = \{t =1\}$ in $(M, g)$, namely a lower bound on its area and upper bound 
on its diameter, 
\be \label{b1}
{\rm area}_g L_1 \geq v > 0, \ \ {\rm diam}_g L_1 \leq D < \infty,
\ee
then one obtains global uniform control on $(M, g, u)$ away from $\dm$. Under the bounds \eqref{b1}, 
any sequence of static vacuum solutions $(M, g_i, u_i) \in \cE$ has a subsequence which converges, in $C_{\d}^{\infty}$ and away from the 
boundary $\dm$ and modulo AF diffeomorphisms, to a limit AF static vacuum solution $(M, g, u)$, cf.~\cite{A2}, \cite{AK} for further 
details. In other words, one has good compactness properties away from the boundary - irrespective of the boundary conditions.

  Thus, the issue is the behavior of solutions arbitrarily near the boundary $\dm$. 

\medskip 

   The conjecture that $\Pi_B$ in \eqref{bmap} is a bijection (Conjecture II) suggests the closely related conjecture that $\Pi_B$ is a 
diffeomorphism. However, this turns out not to be the case. It was observed by the author in \cite{AK} that $\Pi_B$ is not a proper map, and 
so in particular is not a homeomorphism. The reason for this is simple: the structure of a manifold-with-boundary $(M, \dm)$ may 
degenerate with uniform control on the boundary data $(\g, H)$. In other words, the embedded boundary $\dm$ in $M$ 
may ``degenerate" from an embedding to an immersion. 

   As a simple illustration, consider for instance any fixed static vacuum metric $(M, g, u)$ with boundary $\dm$ and boundary data 
$(\g, H) \in \cB$. Let $\Si$ be any embedded $2$-sphere $S^{2} \subset M$ homologous to the boundary $\dm$ with $H_{\Si} > 0$. One 
may then cut off $M$ at $\Si$ to obtain a new exterior solution $M'$ with $\partial M' = \Si$ and 
corresponding boundary data $(\g_{\Si}, H_{\Si}) \in \cB$ at $\Si$. Suppose however $\Si$ is an {\it immersed} $2$-sphere in $(M, g)$ 
homologous to $\dm$. This again has induced ``boundary" data $(\g_{\Si}, H_{\Si})$ from $(M, g)$, but $\Si$ is not the natural smooth boundary 
of an exterior manifold-with-boundary $M'$. If the immersed surface $\Si$ is isotopic to the boundary $\dm$ in $M$ with $H_{\Si} > 0$, 
then at least in many cases one may easily find curves of spheres $\Si_{s}$ in $(M, g)$ which smoothly deform an embedded sphere 
to an immersed sphere, with boundary data $(\g_{s}, H_{s})$ smoothly controlled for all $s$ and with $H_s$ remaining uniformly positive. 
In this process, the ambient metric $g$ and potential $u$ remain fixed up to $\dm$; only the domain $M_s$ on which they are 
defined is changing. Thus, control of the ``boundary data" $(\g_s, H_s)$ does not control the structure of the manifold-with-boundary $M_s$. 
On the other hand, when $(M, g)$ is an extension of a smooth interior body $(\O, g_{\O})$, the corresponding interior geometry 
$(\O_s, g_{\O,s})$ is well-controlled and does not degenerate at all.

\begin{remark}  
{\rm It was proved in \cite{AJ} that this type of degeneration leads to counterexamples to the Bartnik minimization conjecture, Conjecture I. 
There exists a large family of immersed spheres $S^2$ in $\bR^3$ whose boundary data $(\g, H) \in \cB$ are not realized by a static vacuum 
extension of minimal mass; cf.~also \cite{A7} for further discussion and conjectures. 
}
\end{remark}

\begin{remark}
{\rm   One immediate consequence of this degeneration is a kind of symmetry breaking of the gauge group. As discussed in \S 2, the 
natural gauge group for the Bartnik boundary value problem is the group ${\rm Diff}_{1}(M)$ of AF diffeomorphisms 
of $M$ equal to the identity on the boundary $\dm$. This corresponds to isometric metrics, and preserves 
the boundary data $(\g, H)$. Suppose then $(\g, H)$ are realized as boundary data of a static vacuum solution 
$(M, g, u)$, with $\dm$ embedded. Then clearly so are $(\f^{*}\g, \f^{*}H)$, for any diffeomorphism $\f$ of 
$\dm \simeq S^{2}$. Namely any such diffeomorphism extends to an AF diffeomorphism, also called $\f$, of $M$ 
and the triple $(M, \f^{*}g, \f^{*}u)$ is a solution of the static vacuum Einstein equations with boundary data $(\f^{*}\g, \f^{*}H)$. If however 
$\Si$ is merely an immersed ``boundary" in a larger space $M$, then this may not be the case. It is 
no longer clear that if $(\g, H)$ are the boundary data of $\Si$, then $(\f^{*}\g, \f^{*}H)$ are realized as 
boundary values of a static vacuum solution, for general diffeomorphisms $\f$ of $\Si \simeq S^{2}$. 
Thus the symmetry group of the problem has been reduced in the immersed case. 

  This breakdown of diffeomorphism invariance of the problem implies that one must work in a 
suitable fixed gauge. One can no longer deal with metrics up to isometry (i.e.~moduli spaces of 
metrics or solutions to the equations) but only with metrics in a fixed gauge, i.e.~(local) coordinate 
system. It is worth noting that the Weyl metrics, discussed in detail in \S 4, do have a preferred global gauge (coordinate system). 
}
\end{remark}

 Let ($M, g_i, u_i) \in \cE$ be a sequence of static vacuum solutions for which the boundary data $(\g_i, H_i) \in \cB$ converge to a limit 
 $(\g, H)$, $H \geq 0$ in $\w \cB$. There are only two ways that the sequence can degenerate, i.e.~fail to have a subsequence which 
 converges to a limit $(M, g, u) \in \w \cE$, cf.~\eqref{PiB2}). First the structure of an exterior manifold with boundary may break down as 
 above; this occurs only if there is no uniform lower bound on the distance to the cut locus of the $g_i$-normal exponential map of $\dm$ into 
 $(M, g_i)$. Second, the norm of the Riemann curvature ${\rm Rm}$ of $g_i$ may tend to infinity at or arbitrarily near  $\dm$, 
\be \label{Rinf}
|{\rm Rm}_{g_i}(x_i)| \to \infty,
\ee
for some sequence such that $t(x_i) \to 0$, (cf.~\cite{AK}). Of course it is possible that both of these behaviors occur simultaneously. 
We note that a special case of the analysis given in \cite{AK} shows that \eqref{Rinf} does not occur if there are uniform lower and 
upper bounds on the potential $u$ at $\dm$, i.e.~constants $m, M$ such that 
$$0 < m \leq u_i \leq M < \infty \ \ {\rm on} \ \ \dm.$$

  There is one, seemingly natural, condition that can be used to rule out both of these possibilities of degeneration at once. Thus, recall that the 
boundary $\dm$ is called outer-minimizing in $(M, g)$ if for any surface $S \subset M$ surrounding $\dm$, i.e.~$S$ is homologous to $\dm$ 
in $M$, one has 
 \be \label{om}
 \area_g S \geq \area_g \dm.
 \ee
 The boundary $\dm$ is strictly outer-minimizing if strict inequality holds in \eqref{om}, for all $S \neq \dm$. 
 Note that by the first variational formula for area, \eqref{om} implies \eqref{H}, $H \geq 0$ at $\dm$. It is obvious but important to 
 note that the outer-minimizing condition depends on the global structure of the solution $(M, g)$; it cannot be expressed only in terms 
 of the boundary data $(\g, H)$.

  The following result was proved by the author in \cite{AK}. 
  
\begin{theorem} Let $(M, g_i, u_i) \in \bar \cE$ be a sequence of static vacuum Einstein metrics with boundary data $(\g_i, H_i)$ with 
$H_i \geq 0$, $m \geq 4$, such that 
 $$(g_i, H_i) \to (\g, H),$$
in $\w \cB$. Suppose also that $\dm$ is outer-minimizing in $(M, g_i)$ for each $i$:
  $$area_{g_i}S \geq area_{g_i}\dm,$$
for  all $S$ surrounding $\dm$. 
Then a subsequence converges in $\w \cE$ to a limit static vacuum Einstein metric $(M, g, u) \in \bar \cE$ with boundary data $(\g, H)$. 
\end{theorem} 

  Thus, one has good compactness properties under an outer-minimizing condition. 
  
\medskip 
  
    In light of Theorem 3.3, it might seem beneficial to restrict to static vacuum solutions with (strictly) outer-minimizing boundary. Let then 
$\cE^{out} \subset \cE$ be the space of static vacuum solutions for which the boundary is strictly outer-minimizing:
\be \label{stout}
{\rm area}_g S > {\rm area}_g \dm,
\ee
for $S \neq \dm$. This is clearly an open condition so that $\cE^{out}$ remains a smooth Banach manifold. The map 
\be \label{Piout}
\Pi_B : \cE^{out} \to \cB,
\ee
is still well-defined, smooth and Fredholm, of index $0$. 

  In fact it is not unusual in the literature to modify the definition of the Bartnik mass by replacing the no horizon condition (no minimal 
surfaces surrounding the boundary) by an outer-minimizing condition. Thus define 
\be \label{bmo}
m_{B}^{out}(\O, g_{\O}) = \inf  \{m_{ADM}(M, g): (M, g) \in \cP_{\O}^{out}\},
\ee
where $\cP_{\O}^{out}$ is the space of strictly outer-minimizing extensions of $(\O, g_{\O})$ or equivalently of the boundary data $(\g, H)$; 
cf.~\cite{J} for a detailed discussion. Similarly, the minimizing and static extension Conjectures I and II are then 
rephrased to assert the existence and uniqueness of static vacuum solutions with strictly outer-minimizing boundary. One advantage of 
this modified definition is the fundamental estimate of Huisken-Illmanen \cite{HI}, 
$$m_B^{out}(\g, H) \geq m_{H}(\g, H) = \sqrt{\frac{{\rm area}_g \dm}{16\pi}}(1 - \frac{1}{16\pi}\int_{\dm}H^2),$$
comparing the outer-minimizing Bartnik mass with the Hawking mass of the boundary. 
  
  The problem with these reformulations say of Conjecture II is that one needs the strict inequality \eqref{stout} to have a satisfactory 
manifold structure to the space $\cE^{out}$ of solutions. Working instead with the (weak) outer-minimizing condition leads, at best, 
to the structure of a Banach manifold with boundary; this is not suitable  to try to apply general existence methods, such as 
the method of continuity or the more general degree theory techniques discussed briefly in \S 2. 

   On the other hand, in passing to limits of minimizing sequences, (assuming such limits exist), the strict outer-minimizing property is not preserved; 
limits can be expected to be only (weakly) outer-minimizing. Thus for a limit minimizer to be an admissible extension, one needs to work with the 
(closed) weak outer-minimizing condition. In a similar vein, the map $\Pi_B$ in \eqref{Piout} cannot possibly be proper. 
One can easily find sequences $(M, g_i, u_i) \subset \cE^{out}$ with $(\g_i, H_i) \to (\g, H) \in \cB$ smoothly, but which 
have no limit $(M, g, u)$ in $\cE^{out}$. In fact it is easy to find such sequences of boundaries in a fixed background solution $(M, g, u)$) 
with limit only weakly outer-minimizing 
$$\area_g S \geq \area_g \dm.$$
Thus, neither the weak or strict outer-minimizing condition seems satisfactory. We are not aware of any way out of this quandary. 
    
\begin{remark}  
{\rm  One might raise the same objection with the no-horizon condition itself; if a static vacuum solution had a horizon, (i.e.~a minimal 
surface surrounding $\dm$), then it could not  be an admissible extension and so could not realize the Bartnik mass as in Conjecture I. 
One would have to consider only static vacuum solutions without horizons, which would be very awkward. Fortunately, 
Proposition 2.1 rules out this situation. 
 
  }
  \end{remark} 

   Of course, the main trouble here is that the outer-minimizing property (strict or weak) cannot be phrased in terms of the boundary data $(\g, H)$. 
There has been some speculation that suitable conditions on the boundary data $(\g, H)$ would ensure that static vacuum solutions 
with such boundary data have strict outer-minimizing boundary; a common suggestion, based on the resolution of the Weyl embedding problem, 
is the condition 
$$K_{\g} > 0 \ \ {\rm and} \ \ H > 0.$$
However, the next result shows that this cannot possibly work, at least for $H$ small. 
   
  Using Theorem 3.3, we prove the following non-existence result. This may be viewed as a stable version of the black hole uniqueness theorem, 
as in Proposition 2.1, with $H$ small replacing $H = 0$. 

\begin{theorem}  Let $\g$ be a smooth metric on $S^{2}$ such that the Gauss curvature $K_{\g}$ is a Morse function 
in the domain $K_{\g} > 0$, i.e.~$K_{\g}$ has only non-degenerate critical points when $K_{\g} > 0$.  

  Let $\mu(p) = \min |\l_i(p)|$, where $\l_i(p) \neq 0$ are the eigenvalues of the Hessian $D^2K_{\g}$ at a critical point $p$ of $K_{\g}$, 
$K_{\g}(p) > 0$. Let also $C = ||\g||_{C^{m,\a}} + ||H||_{C^{m-1,\a}}$, for some $m \geq 4$. 

   Then there is a constant $H_0 > 0$, depending only on $C$ and a positive lower bound for $\mu_0 = \min_p \mu(p) > 0$, such that:  
for any smooth function $H$ satisfying
\be \label{H0}
0 < H < H_{0},
\ee
pointwise, the Bartnik boundary data $(\g, H)$ are not the boundary data of a static vacuum Einstein metric $(M, g, u) \in \cE$. Thus 
for such $(\g, H) \in \cB$, 
$$\Pi_B^{-1}(\g, H) = \emptyset.$$
\end{theorem}

   We note that at least a large class of boundary data satisfying \eqref{H0} above do have positive scalar curvature fill-ins, i.e.~are 
boundary data of compact domains $(\O, g_{\O})$ with positive scalar curvature. 

\begin{proof} The proof is by contradiction. Thus, let $(\g_{i}, H_{i}) \in \cB$, ($H_{i} > 0$), be a sequence of boundary data 
converging in $\w \cB$ to $(\g, 0) \notin \cB$ and satisfying the assumptions of Theorem 3.5. Suppose there is a sequence of 
static vacuum solutions $(M, g_{i}, u_{i})$, $u_{i} > 0$ with Bartnik boundary data $(\g_i, H_i)$. We then need to derive a 
contradiction. 

  In case the boundary $\dm$ is (weakly) outer-minimizing in $(M, g_{i})$ for $i$ sufficiently large, Theorem 3.5 is proved in \cite{A7}, 
based on the discussion in \cite{AJ}. We sketch the argument for completeness. Since $\dm$ is outer-minimizing, Theorem 3.3 implies 
that a subsequence of $(M, g_{i}, u_{i})$ converges in $\w \cE$ to a limit solution $(M, g, u) \in \w \cE$ of the static vacuum Einstein 
equations with $u \geq 0$ and $u \to 1$ at infinity.  The Bartnik boundary data 
of $(M, g, u)$ is given by 
$$\lim_{i\to \infty} (\g_{i}, H_{i}) = (\g, 0).$$

  By the black hole uniqueness theorem as in Proposition 2.1, the only static vacuum solution $(M, g, u)$ with 
$u \geq 0$, $u \to 1$ at infinity and smooth horizon boundary $H = 0$ is the Schwarzschild metric with boundary data $(\g_{2m}, 0)$. 
Since the Gauss curvature $K_{\g_{2m}} \equiv (2m)^{-2}$ is not a Morse function, this gives a contradiction. Namely the smooth convergence 
of $(M, g_i, u_i)$ up to the boundary implies that $K_i$ is smoothly close to the constant $(2m)^{-2}$; however, the eigenvalues of the 
Hessian $D^2K_i$ are bounded away from $0$, say at $\max K_i$. 

  Now in general, consider a static vacuum solution $(M, g, u)$ (say equal to some $(M, g_i, u_i)$ above) with boundary data $(\g, H)$ with 
$H > 0$ but close to $0$. Let $\Si$ be the outer-minimizing hull of $\dm$ in $M$, cf.~\cite{HI}. This is the embedded surface surrounding 
$\dm$ in $(M, g)$ of smallest area. Note first that by Proposition 2.1, $\Si \subset M$ cannot be disjoint from $\dm$, so 
$$\Si \cap \dm \neq \emptyset.$$
Let $U \subset \dm$ be the interior of $\dm \cap \Si$ and let $V = \Si \setminus \bar U \subset \Si$. Then (again by Proposition 2.1), 
both $U$ and $V$ are non-empty open sets; of course, neither $U$ nor $V$ is assumed to be connected. By the outer-minimizing property, one has 
$$H = 0 \ \ {\rm on} \ \ V.$$
The surface $\Si$ is $C^{1,1}$ at the seam or corner $\p U \cap \p V$, smooth away from $\p U \cap \p V$ and has non-negative 
distributional mean curvature $H$ across $\p U \cap \p V$, cf.~\cite{HI}. 

  Now given $(M, g_i, u_i)$ as above, consider the new sequence of static vacuum solutions $(\hat M_i, g_i, u_i)$ where 
$\hat M_i \subset M$ is the region of $M$ exterior to $\Si_i$, so $\partial \hat M_i = \Si_i$. The analysis to follow will be to show that 
that the same arguments as before where $\dm_i$ was outer-minimizing can be applied to the new sequence $(\hat M_i, g_i, u_i)$ to 
again obtain a contradiction. 
  
  Let $\g_{U_i} = g_i |_{U_i}$ and $\g_{V_i} = g |_{V_i}$. The metric $\g_{U_i}$ is uniformly controlled in $C^{m,\a}$ by hypothesis, 
since $U_i \subset \dm_i$. Now assume first, for simplicity, that the sequence $(V_i, \g_{V_i})$ is uniformly bounded in $C^{4,\a}$ 
(modulo diffeomorphisms), so that the Gauss curvature $K_i$ of $(V_i, \g_{V_i})$ is uniformly bounded in $C^{2,\a}$. The same 
then holds for the induced metric $\hat \g_i$ on $\p \hat M_i$, away from the seam or corner $\p V_i \cap \p U_i \subset \Si_i$. 
Thus away from the corners, one has $H_i \to 0$ in $C^{3,\a}$ with $\hat \g_i$ bounded in $C^{4,\a}$. 
One may then still apply the proof of Theorem 3.3 in this setting to conclude that the curvature ${\rm Rm}_{g_i}$ and its derivatives 
up to order $2+\a$ are uniformly bounded in $\hat M_i$, up to the boundary $\p M_i$. Thus, in a subsequence, $(\hat M_i, \hat g_i, \hat u_i)$ 
converges in $C^{4,\a}$ away from corners to a round Schwarzschild metric, cf.~also the proof of Proposition 2.1 and Remark 2.2. In 
particular, this convergence holds in $U_i \subset \dm_i$ and so gives the same contradiction as above - even if components of $U_i$ 
shrink to points for instance. 

  Now the surfaces $(V_i, \g_{V_i})$ are stable minimal surfaces in $(M, g_i)$. In the Appendix, we prove a regularity result for such stable 
minimal surfaces in static vacuum Einstein manifolds which implies that if both the scalar curvature $s_{V_i}$ of $(V_i,  \g_{V_i})$ and the 
full Riemann curvature ${\rm Rm}$ of $(M, g_i)$ are uniformly bounded along $V_i$,  
 \be \label{Kbound}
 |s_{V_i}| + |{\rm Rm}_{g_i}| \leq C
 \ee
 on $V_i$, then in fact $(V_i, \g_{V_i})$ is uniformly bounded in $C^{\infty}$ modulo diffeomorphisms. Note this statement is standard for 
the trivial static vacuum Einstein metric $\bR^3$. It also follows easily in regions of $V_i$ where $t_i = dist_{g_i}(\dm_i, \cdot)$ is bounded 
away from $0$, by the estimates \eqref{apri}. Thus, the proof follows just as above if \eqref{Kbound} holds, and so it suffices to 
prove \eqref{Kbound}. 

  We prove \eqref{Kbound} by a blow-up argument. The discussion below applies to each $(\hat M_i, g_i, u_i)$ but we drop the index $i$ from the 
notation; the estimates below are then understood to hold uniformly, for all $i$ large. 

  To begin, the $2^{\rm nd}$ variation of area gives 
\be \label{2nd2}
\int_{\Si}(|df|^2 + (N(H)+ H^2)f^2) dv_{\g_{\Si}} \geq 0, 
\ee
and the normal variation of $H$ is given by $N(H) = -|A|^2 - {\rm Ric}(N,N)$. By the Hamiltonian constraint (Gauss) equation \eqref{Hamcon},  
$|A|^2 - H^2 + s_{\g} = -2{\rm Ric}(N,N)$, 
so that 
$$N(H) = -{\tfrac{1}{2}}(|A|^2 + H^2 - s_{\g_{\Si}}).$$
Choosing $f = 1$ in \eqref{2nd2} thus gives the bound 
$$\int_{\Si}|A|^2 \leq \int_{\Si}s_{\g_{\Si}} + H^2  = 4\pi \chi(\Si) + \int_{\Si}H^2.$$
Since $H^2 \to 0$ uniformly, it follows that one has the uniform bound 
\be \label{Ab}
\int_{\Si}|A|^2 \leq 8\pi + 1.
\ee
Note this bound is scale invariant. Next, again from the Gauss equation, cf.~\eqref{Hamcon2}, one has 
$$|A|^2 - H^2 + s_{\g} = 2u^{-1}(\D u + HN(u)).$$
Integrating over $\Si$ and applying divergence formula gives 
$$\int_{\Si}|d^T \nu|^2 \leq 8\pi - \int_{\Si}HN(\nu),$$
where $d^T\nu$ is the tangential gradient of $\nu$ on $\Si$. On the domain $V \subset \Si$, $H = 0$ while on the locally 
outer-minimizing complement $U \subset \dm$, $HN(\nu)$ converges to a well-defined limit by Theorem 3.3, cf.~\cite{AK} 
for details. It follows that 
\be \label{nub}
\int_{\Si}|d^T\nu |^2 \leq C,
\ee
for some uniform $C$, independent of $i$. 

  Next, for the blow-up argument, choose points $p_i \in \Si_i$ realizing the maximum of 
\be \label{Q}
Q_i(x)= |s_{\g_{V_i}}(x)|  + |{\rm Rm}_{g_i}(x)|,
\ee
$x \in \Si_i$. Without loss of generality, assuming $Q_i (p_i)\to \infty$, we rescale the metrics $g_i$ (and correspondingly $\g_{V_i}$) by 
$$\bar g_i = Q_i(p_i) g_i,$$
so that 
\be \label{Qmax}
\bar Q_i \leq 1 \ \ {\rm and} \ \ \bar Q_i (p_i)= 1,
\ee 
on $\Si_i$. Clearly one has $p_i \in V_i$. We also renormalize $u$ to $\bar u_i = u_i / u(p_i)$. Again by (the proof of)  
Theorem 3.3, $|{\rm Rm}_{\bar g_i}|$ is uniformly bounded on $(\hat M_i, \bar g_i)$. 
 
   To prove that a subsequence $(\hat M_i, \bar g_i, p_i)$ converges modulo diffeomorphisms, (i.e.~in the Cheeger-Gromov sense), 
in $C^{1,\a}$, we need to rule out the possibility of collapse of the metrics, i.e.~the collapse of the volume of unit balls at $p_i$. 

\begin{lemma} (Non-collapse). For the rescaled sequence $(\hat M_i, \bar g_i, p_i)$, there is a uniform lower bound 
\be \label{v1}
{\rm vol}_{\bar g_i}B_{p_{i}}(1) \geq v_0 > 0.
\ee
Similarly, the boundary geometry $(\Si_i, \bar \g_i, p_i)$ does not collapse: 
\be \label{v2}
{\rm area}_{\bar \g_i}D_{p_i}(1) \geq a_0 > 0.
\ee 
\end{lemma} 

\begin{proof} 
First, as discussed in the Appendix, cf.~\eqref{tilde}, for any $(M, g, u) \in \cE$, the conformally equivalent metric $\w g = u^2 g$ 
has positive Ricci curvature. The metric $\w g$ is AF and hence by the Bishop-Gromov volume comparison theorem, there is a 
uniform scale-invariant lower bound on the volume of $r$-balls in $(M, \w g)$:
$$\frac{{\rm vol}_{\w g}B(r)}{r^3} \geq v_0 > 0.$$
Now apply this to the rescaled metrics $(\hat M_i, \bar g_i, \bar u_i)$ based at $p_i$. Since $\bar u_i$ is uniformly bounded 
above and below in $(B_{p_i}(1), \bar g_i)$, (cf.~\eqref{up2}), the metrics $\bar g_i$ and $\w g_i$ are uniformly quasi-isometric in this region. 
This proves \eqref{v1}. 

    To see that the boundary geometry also cannot collapse, recall that $(\Si_i, \bar \g_i)$ has uniformly bounded (Gauss) 
curvature. If the geometry is collapsing at $p_i$, then the geometry of a large geodesic disc $D_{p_i}(R)$ about $p_i$ 
in $(\Si_i, \bar \g_i)$ is that of a long cylinder which has a foliation by short (geodesic) loops. Since the ambient curvature is 
uniformly bounded, if $\ell$ denotes the length of the short loops in $(\Si_i, \bar \g_i)$ near $p_i$, then $|A| \sim \ell^{-1}$ and hence 
 $$\int_I |A|^2 dv_{\g} \sim \int_I \ell \ell^{-2} \sim \ell^{-1}.$$
 Since $A$ is bounded in $L^2$ by \eqref{Ab} and since \eqref{Ab} is scale-invariant, this gives a uniform lower bound on $\ell$, 
 proving \eqref{v2}. 
 
 \end{proof}
   
  The discussion above proves that, after passing to a subsequence, the sequence $(\hat M_i, \bar g_i, \bar u_i, p_i)$ converges in the 
pointed $C^{1,\a}$ topology, modulo diffeomorphisms, to a complete non-compact limit $(M_{\infty}, \bar g_{\infty}, \bar u_{\infty}, p)$, with 
complete, non-compact boundary $(S, \g_{S}, \bar u_{\infty}, p)$. The rest of the argument to follow is to prove that the limit is flat, 
with flat boundary $S$. By results in the Appendix, cf.~Proposition 5.2, the convergence to the limit is in the strong $C^2$ topology, 
giving a contradiction to the continuity of $Q$ in \eqref{Qmax} under $C^2$ convergence. This will complete the proof of Theorem 3.5. 

  All the discussion to follow takes place on the limit $M_{\infty}$ and $S$ so we drop the bar and $\infty$ below to simplify the notation. 
     
   On the blow-up limit boundary $S$, we have $A$ and $d\nu$ are in $L^2(S, \g_S)$, by \eqref{Ab} and \eqref{nub}. Moreover, the 
ambient curvature $|{\rm Rm}|$ is uniformly bounded along $S$. As above, $D(r)$ denotes a geodesic $r$-disc in $(S, \g_S)$ about $p$. 
  
 \begin{lemma}
 The limit $(S, \g_S)$ has at most quadratic area growth, i.e.
 \be \label{quad}
 {\rm area}_{\g_S} D(r) \leq V r^2,
 \ee
 for some $V < \infty$. 
 \end{lemma}
 
 \begin{proof} 
 On $(S, \g_S)$, we have 
\be \label{HamS}
 |A|^2 + s_{\g_S} = 2u^{-1}\D u.
 \ee
 Let $v(r)$ be the length of the boundary $S(r) = \p D(r)$. Integrating \eqref{HamS} over $D(r)$ and applying the Gauss Bonnet 
 and divergence theorems gives 
$$ 4\pi - 2v'(r) = -\int_{D(r)} |A|^2 + 2\int_{D(r)}|d^T\nu|^2  + 2\int_{S(r)} \partial_r \nu.$$
By \eqref{Ab}, this gives 
$$v'(r) \leq C + \int_{S(r)} |d^T\nu | .$$
On intervals where $v'(r) \leq C$, \eqref{quad} holds (by integration), so assume only 
$$v'(r) \leq C \int_{S(r)} |d^T\nu | .$$
By the H\"older inequality, $\int_{S(r)}|d^T\nu| \leq (\int_{S(r)} |d^T\nu|^2)^{1/2}v(r)^{1/2}$, so that 
$((\sqrt{v'(r)})')^2 \leq C \int_{S(r)}|d^T \nu|^2$. Let $q(r) = \sqrt{v'(r)}$, so that via \eqref{nub} we obtain 
$$\int_0^r (q')^2 \leq C.$$
Then $q(r) \leq \int_0^r q'(r) \leq (\int_0^r (q'(r)^2)^{1/2}r^{1/2} \leq Cr^{1/2}$, so that again $v(r) \leq C r$. 
This proves the result.   

 \end{proof}
 
\begin{lemma} 
The surface $(S, \g_S) \subset (M, g)$ is AF in the weak sense that 
$$Q(x_k) \to 0,$$
on any sequence $x_k \to \infty$ in $S$. 
\end{lemma} 

\begin{proof} 
Let $x_k$ be any divergent sequence in $S$ and consider the geometry of pointed manifolds $(S, \g_S, x_k)$ in regions about $x_k$. 
By Lemma 3.6, this sequence does not collapse, and so has a $C^{1,\a}$ limit (in a subsequence). Also, $A \to 0$ and $d\nu \to 0$ in 
$L_{loc}^2$, by \eqref{Ab} and \eqref{nub}. By the regularity estimates in the Appendix, $A$ is uniformly bounded in $L^{\infty}$, 
as is $u \in L_{loc}^{2,p}$ on the boundary $S$ and in $M$, cf.~\eqref{up1}, \eqref{up2}. Hence $N(u)$ is also uniformly bounded. 
By the constraint equation \eqref{dnu}, 
$$\d A - A(d\nu) = u^{-1}dN(u).$$
The term $A(d\nu) \to 0$ in $L_{loc}^2$, while the term $\d A \to 0$ in $L_{loc}^{-1,2}$. It follows that $dN(u) \to 0$ in $L_{loc}^{-1,2}$ 
and hence $N(u) \to const$ weakly in $L_{loc}^2$. Since $N(u)$ is also uniformly bounded in $L^{\infty}$, $N(u) \to const$ strongly in 
$L_{loc}^2$. It follows that on the limit (which is smooth by regularity results from Appendix), one has the limit data
$$A = 0, \ \ u = const, \ \ N(u) = const.$$
Similarly, $NN(u) = 0$ on the limit, etc. Since $\D_g u = 0$ on $M$, unique continuation for the Laplacian implies that $D^2 u = 0$ 
on $(M, g)$ so that $u$ is an affine function and the metric $g$ is flat. Since $A = 0$, the boundary metric $\g$ is also flat. 

  It follows that $Q \equiv 0$ on the limit. By the strong convergence (Proposition 5.2) in the Appendix, it follows that $Q(x_k) \to 0$, 
as claimed. 
  
 \end{proof}  
  
    It follows, by repeated use of Lemma 3.8 at various blow-down scales that all tangent cones at infinity of $(S, \g_S)$ are flat, with flat 
ambient geometry, i.e.~$(D(r), \frac{1}{r^2} \g_S)$ converges to a flat metric away from the origin $p$ and similarly with $g_S$ in 
place of $\g_S$. Now apply the stability inequality \eqref{2nd2} on the complete $(S, \g_S)$.  By the well-known log-cutoff trick, 
(cf.~\cite{CM} for example) one has 
 $$\int_S |A|^2 - s_{\g_S} \leq 0,$$
 so that by the Gauss-Bonnet theorem,  
\be \label{GB}
\int_S |A|^2 \leq  \int_S s_{\g} = 2\pi - \lim_{r\to \infty} v'(r).
\ee
This implies first that $\lim_{r \to \infty} \frac{v(r)}{r} \leq 2\pi$. Suppose strict inequality holds, $\lim_{r \to \infty} \frac{v(r)}{r} < 2\pi$. The tangent 
cone at infinity $C$ is $S$ is then a flat cone, of cone angle $\a < 2\pi$, which bounds a flat static vacuum solution. By the scale-invariant 
apriori estimates \eqref{apri}, the limit static vacuum solution is smooth away from the boundary cone $C$, so cannot be of the form 
$C\times \bR^+$, which is singular in the interior: any singularity of the cone $C$ cannot propagate into the exterior region. On the 
other hand, if $C$ is a singular cone, $\a < 2\pi$, in an ambient smooth flat geometry, then one has $\int_C |A|^2 = \infty$, again a 
contradiction. (Alternately, such cones are not outer-minimizing in the ambient geometry). 

   It follows that 
\be \label{2pi}
\lim_{r \to \infty} \frac{v(r)}{r} = 2\pi,
\ee  
and hence by \eqref{GB}, 
$$\int_S |A|^2 = 0,$$
so that $A = 0$ on $S$. The potential equation \eqref{HamS} then becomes 
$$K = u^{-1}\D u.$$
Integrating this over $D(r)$ and taking the limit gives 
$$\lim_{r \to \infty} [\int_{D(r)} |d^T\nu|^2 + \int_{S(r)}\p_r \nu ] = 0,$$
since $\lim_{r \to \infty}\int_{D(r)}K = 0$. As in the proof of Lemma 3.7, this implies 
$$\int_{S} |d^T\nu|^2 = 0,$$
so that $u = const$ on $S$. Thus the Cauchy data $(\g, A, u, N(u))$ for the static vacuum Einstein equations equal that of a flat solution 
with affine potential. By unique continuation as in the proof of Lemma 3.8, it follows that the limit is flat, with flat boundary, so that 
$Q = 0$ on $S$. This again contradicts the strong convergence to the limit, cf.~Proposition 5.2. 
  
\end{proof}

  The condition that $K_{\g}$ is a Morse function is a sufficient but we don't believe a necessary condition. It is natural to conjecture 
that Theorem 3.5 holds for all $\g$, with then $H$ sufficiently small depending only on lower and upper bounds for the distance of 
$\g$ to a round metric in say $C^{4,\a}$. 

   It follows in particular that the mass $m_{B}$ cannot be realized by boundary data $(\g, H)$ with $H$ sufficiently small and $\g$ not 
a round metric, i.e.~Conjecture I also fails in this situation. 

\begin{remark}
{\rm  Theorem 3.5 should be contrasted with the following result, proved in \cite{A6}; the proof also strongly uses the outer-minimizing property: 
for any given boundary data $(\g, H) \in \cB$, there is a $\l > 0$ such that 
$$(\g, \l H) \in {\rm Im}\, \Pi_B,$$
i.e.~there is a static vacuum solution $(M, g, u)$ with boundary data $(\g, \l H)$. Moreover, $\dm$ is outer-minimizing in $(M, g)$. 
Thus, for any $\g$, if one increases $H$ sufficiently by a positive multiplicative factor, then there is a solution to the static vacuum equations 
realising the given Bartnik boundary data. 
}
\end{remark}

   For boundary data $(\g, H) \notin {\rm Im}\, \Pi_B$ as in Theorem 3.5, a minimizing sequence for $m_B(\g,H)$ cannot converge, 
so it must degenerate. How? Consider first the boundary data $(\g, 0) \in \w \cB$, with $\g$ not a round metric on $S^2$, compare 
with \eqref{Schdeg}. The black hole uniqueness theorem (Proposition 2.1) implies that a minimizing sequence for this boundary data must 
degenerate. The work of Mantoulidis-Schoen \cite{MS} indicates how this occurs -- when the first eigenvalue  $\l_1(-\D_{\g} + K_{\g}) > 0$ -- for 
minimizing sequences with respect to the outer-minimizing mass $m_B^{out}(\g, 0)$ in \eqref{bmo}. Briefly, the AF end of a minimizing sequence 
converges to a Schwarzschild metric $g_{Sch(m)}$ up to or  arbitrarily near the horizon, where $m$ is given by 
 $$m = \sqrt{\frac{\area_{\g}\dm}{16\pi}}.$$
 The remaining portion of the sequence forms an (arbitrarily) long cylinder of non-negative scalar curvature, connecting the nearly 
 round metric near the horizon at one end to the metric $\g$ at the other end. The cylindrical region is 'hidden behind' the Schwarzschild 
 horizon; since this region does not contribute to the ADM mass, one has no effective control on the behavior of a minimizing sequence 
 in this region. We refer also to the work of Miao-Xie \cite{MX} for related results. 
 
   We conjecture similar behavior occurs for minimizing sequences for $m_B(\g, H)$ when $H > 0$ is sufficiently small, depending on $\g$ 
as in Theorem 3.5. 

\medskip 
  
The non-existence region above strongly suggests there is a large region of fold behavior for the map $\Pi_B$, which, in turn, implies a large 
region of non-uniqueness for Conjecture II (compare with the discussion in and preceding the toy model \eqref{S1}).  

 We describe the situation on a concrete example. Choose a fixed Schwarzschild metric $g_{Sch}(m)$ and choose a round sphere $S(t)$ at fixed 
distance $t > 0$ to the horizon. Let $(\g_0, H_0)$ be the induced Bartnik boundary data on $S(t)$, so $\g_0$ is a round metric and $H_0 > 0$ is a 
small constant. It is proved in \cite{AK} that ${\rm Ker}\, D\Pi_B = 0$ at the Schwarzschild metric exterior to the horizon $H = 0$ (even though 
$D\Pi_B$ is not Fredholm there); hence ${\rm Ker} \, D\Pi_B = 0$ at the exterior region to $S(t)$ for $t$ small, i.e.~this exterior metric 
is a regular point of $\Pi_B$. Now choose any fixed metric $\g$ satisfying the hypotheses of Theorem 3.5, i.e.~$K_{\g}$ is a Morse 
function in the region $K_{\g} > 0$. Consider the curve of boundary data 
$$L(s) = (1-s)(\g_0, H_0) + s(\g, H_0) \subset \cB,$$
for $s \in [0,1]$ and the corresponding inverse image 
 $$\Pi_B^{-1}(L(s)).$$
 By the Smale-Sard theorem \cite{Sm}, if necessary one may perturb $L: I \to \cE$ slightly to $\w L: I \to \cE$, keeping the endpoints fixed, so 
 that $\w L: I \to \cE$ is transverse to $\Pi_B$. The inverse image 
$$\Pi_B^{-1}(\w L),$$
is then a collection of 1-dimensional curves $\{\s_i(s)\}$ in $\cE$. If there is more than one component of $\Pi_B^{-1}(\w L)$ then of course one 
already has non-uniqueness for Conjecture II, so suppose then that $\s = \Pi_B^{-1}(\w L)$ is a connected curve. By Theorem 3.5, for $H_0$ 
sufficiently small depending on $\g$, there is $s_0 < 1$ (possibly close to 1) such that either 
\be \label{eor}
\Pi_B^{-1}(\w L[s_0, 1]) = \emptyset \ \ {\rm or} \ \  \Pi_B^{-1}(\w L(s_0, 1]) = \emptyset .
\ee
Without loss of generality, assume $\Pi_B^{-1}(\w L(s)) \neq \emptyset$ for $s < s_0$. We discuss these two possibilities in more detail. 
  
  I. In the first case above, as $s \to s_0$ with $s < s_0$, the curve $\s(s)$ diverges to infinity in $\cE$, i.e.~$\Pi_B$ is not a proper map over $\w L$. 
Since at $s = 0$ the boundary $S(t)$ is strictly outer-minimizing in the Schwarzschild metric, the boundary $\dm$ of $\s(s)$ remains outer-minimizing for $s$ 
sufficiently small. However, by Theorem 3.3, the boundary cannot remain outer-minimizing for all $s \in [0,s_0)$. The divergence of the curve 
$\s(s)$ as $s \to s_0$ is due either to a loss of the manifold-with-boundary structure or the curvature blows up on approach to the boundary (or both). 
  
  II. In the second case, the map $\Pi_B$ is proper at least over $\w L[0, s_1)$ for some $s_1 > s_0$. In this case, $\Pi_B$ exhibits fold behavior 
(like $x \to x^2$) over $\w L$. Such fold points $\s(s_0)$ are critical points of $\Pi_B$ and the map $\Pi_B$ is locally 2-1 over $\w L$ 
near $\w L(s_0)$. This again gives non-uniqueness for Conjecture II.  
  
    The same analysis as above applies to general curves in $\cB$, which start at data $(\g_0, H_0)$ for which $(\Pi_B)^{-1}(\g_0, H_0) \neq \emptyset$ 
and which end at points $(\g, H)$ satisfying the conditions of Theorem 3.5, for which $(\Pi_B)^{-1}(\g, H) = \emptyset$. We conjecture that, at least 
for a large class of curves, Case I does not occur, so the fold behavior of Case II holds.

\section{Weyl metrics} 

  It is of course important to have a large class of examples on which one can test the Bartnik conjectures in various regions of boundary 
data. The most interesting class of (relatively) explicit solutions are the Weyl solutions \cite{W}, \cite{BW}, which have an additional axial 
symmetry.  These metrics have a hypersurface-orthogonal isometric $S^{1}$ action, so that the metric $g = g_{M}$ on $M$ has the form
\be \label{weyl}
g_{M} = f^{2}d\f^{2} + g_{V},
\ee
where the orbit space $(V, g_{V})$ is a Riemannian surface with $V \simeq (\bR^3 \setminus B)/S^1 \simeq (\bR^2)^+ \setminus D$. Thus the 
Ricci flat 4-metric $(\cM, g_{\cM})$ (as in \eqref{st}) has the form
\be \label{cm}
g_{\cM} = \pm u^{2}dt^{2} + f^{2}d\f^{2} + g_{V},
\ee
with $u$, $f$ positive functions on $V$. This gives the 4-manifold $\cM$ the structure of a toric Einstein 4-manifold (when the $t$-factor is 
compactified to $S^1$ in the Riemannian setting). The static vacuum equations \eqref{stat} are then reduced to equations on $(V, g_V)$. 

  As in \eqref{cE}, let 
\be \label{abweyl}
\cE_{S^1}^{m,\a} = \bE_{S^1}^{m,\a} / {\rm Diff}_1^{m+1,\a}(M),
\ee
denote the (abstract) space of isometry classes of such AF Weyl metrics of the form \eqref{weyl}, with $H > 0$ at $\dm$. Similarly, the 
space of $S^1$-invariant boundary data is given by 
$$\cB_{S^1}^{m,\a} = {\rm Met}_{S^{1}}^{m,\a}(S^{2})\times C_{S^{1},+}^{m-1,\a}(S^{2}).$$
Here ${\rm Met}_{S^{1}}^{m,\a}(S^{2})$ denotes the space of $C^{m,\a}$ metrics on $S^2$ of the form $\a^2 d\t^2 + \b^2 d\f^2$, $\a = \a(\t) > 0, 
\b = \b(\t)$ where $(\t, \f)$ are standard spherical coordinates on $S^2(1)$; there is no $d\f d\t$ cross term in \eqref{weyl}. At the poles 
$\t = 0, \pi$, one has $\b = 0$ and $\b' = \pm \a$ with similar higher order conditions on $\b$ at $0$, $\pi$ for $C^{m,\a}$ smoothness. The functions 
in $C_{S^{1},+}^{m-1,\a}(S^{2})$ are positive functions $H$ of the form $H = H(\t)$ with similar smoothness conditions at the poles.

    The associated Bartnik boundary map 
\be \label{Bart3}
\Pi_B: \cE_{S^1} \to \cB_{S^1},
\ee
$$(M, g, u) \to (\g, H),$$
is still a smooth Fredholm map of Fredholm index 0. One expects or at least hopes that it should be considerably simpler to understand the 
validity of Conjectures I and II in the context of Weyl metrics compared with the general case. 

   An important feature is that Weyl metrics have a canonical choice of coordinates, Weyl cylindrical coordinates and it is useful to derive this 
in some detail, cf.~also, \cite{A2} \cite{Kr} for background material on these metrics. First, note that both $M_{u} = V\times_{f}S^{1}$ with metric 
$g_{V} + f^{2}d\t^{2}$ and $M_{f} = V\times_{u}S^{1}$ with metric $g_{V} + u^{2}d\t^{2}$ are static 
vacuum solutions on the respective $3$-manifolds, with potentials $u$ and $f$. Thus one has a natural 
dual pairing $u \leftrightarrow f$, cf.~\cite{A2} for further discussion. On $M_{u}$,  
$$0 = \D_{M_{u}}u = \D_{V}u + \<d\log f, du\>,$$
while on $M_{f}$, 
$$0 = \D_{M_{f}}f = \D_{V}f + \<d\log u, df\>.$$
It follows that the function  
\be \label{r}
r = fu
\ee
representing the area of the toral fibers in \eqref{cm}, is harmonic on $V$, $\D_{V} r = 0$. Noting that 
$V$ is simply connected, let $z$ be the harmonic conjugate of $r$ on $V$. Then $(r, z)$ give 
isothermal coordinates for the metric $g_{V}$. Adding the coordinate $\f \in [0,2\pi]$ gives 
the Weyl canonical cylindrical coordinates $(r, z, \f)$ on $M = V \times_{f}S^{1}$. In these coordinates, 
the metric $g$ on $M$ has the form 
\be \label{can}
g = g_{M} = u^{-2}[e^{2\l}(dr^{2} + dz^{2}) + r^{2}d\f^{2}],
\ee
for some function $\l$. Let $\nu = \log u$. Further analysis using the static vacuum equations \eqref{stat} shows that $\l$ satisfies 
\be \label{lam}
\l_{r} = r(\nu_{r}^{2} - \nu_{z}^{2}), \ \ \l_{z} = 2r\nu_{r}\nu_{z},
\ee
so that the 1-form $d\l$ is given by $d\l = r(\nu_{r}^{2} - \nu_{z}^{2})dr + 2r\nu_{r}\nu_{z}dz$. Thus, 
$\l$ is completely determined by $u$, up to a constant.  On the $z$-axis $A$ where $r = 0$ 
(so $\t = 0, \pi$) one has $\l = const$. The metric \eqref{can} is regular (i.e.~smooth) at $A$ only 
when $\l = 0$; otherwise there is a cone singularity along $A$. Thus, $\l$ is uniquely determined by $\nu$. 
Note that $r$ in \eqref{r} is uniquely determined by $(M, g, u)$\footnote{Since $u > 0$ on $M$ and $g$ is smooth 
on $M$, note that \eqref{can} also shows that $r$ is a well-defined coordinate, i.e.~$dr \neq 0$ on $M$.} while $z$ is 
uniquely determined up to a constant; this constant may be fixed by imposing the normalization condition 
\be \label{norm}
\int_{\dm} z \, dv_{\g} = 0.
\ee 
We note that $(r, z)$ normalized as in \eqref{norm} vary smoothly with $(g, u)$. 

The coordinates $(r, z, \f)$ represent standard cylindrical coordinates on $\bR^{3}$. Most importantly, the potential 
function $\nu = \log u$ is an axi-symmetric (i.e.~$\f$-invariant) harmonic function with respect to the {\it Euclidean} 
Laplacian:
\be \label{Eucl}
\D_{Eucl}\nu = 0.
\ee
Observe that the full Weyl solution \eqref{can} is determined by the single potential function $\nu$ in these coordinates. 

 \medskip 

   One main point here is that given any (abstract) Weyl metric, there is a canonical choice of 
coordinates, i.e.~gauge, read off from the geometry of $S^{1}$ fibers, compare with Remark 3.2. In this chart, the Weyl 
solution $(M, g, u)$ is completely determined by the axi-symmetric Euclidean harmonic function $\nu$ and the location 
of the boundary $\dm$. This is the remarkable Weyl reduction of a non-linear system of equations to a linear equation. 

\medskip 

  Working in these coordinates, a boundary is specified by a, say $C^{m+1,\a}$, embedded curve 
\be \label{sig}
\s: [0,\pi] \to (\bR^{2})^{+}
\ee
in $(\bR^{2})^{+}$, $\s(\t) = (r(\t), z(\t))$ meeting the axis $A$ smoothly and orthogonally only at $\t = 
0,\pi$. The rotation of $\s$ about the axis $A$ generates an embedding 
$$F: S^{2} \to \bR^{3}$$ 
of a sphere as a surface of revolution in $\bR^3$. The image $\Si = {\rm Im}\, F$ divides $\bR^{3}$ into a compact interior region 
diffeomorphic to a 3-ball $B$ and a non-compact exterior region diffeomorphic to $\bR^3 \setminus B$. Here we focus exclusively 
on the exterior manifold-with-boundary $M$ with $\dm = \Si$. We will assume that $\s$ is oriented\footnote{This will be dropped 
later.} in that $\t = 0$ corresponds to the north pole of $S^2$ while $\t = \pi$ corresponds to the south pole and the normal vector 
$\p_z$ at $\s(0)$ points into the exterior region $M$. 

The induced metric $\g$ on the boundary is given by
\be \label{gamma}
\g = F^*g = u^{-2} [e^{2\l}((r'(\t))^2 + (z'(\t))^2)d\t^2 + r^2 d\f^2 ].
\ee
By computation, one has 
\be \label{Hform}
e^{\l - \nu}H = H_{Eucl} + N_E(\l - 2\nu), 
\ee
where $N_E$ is the Euclidean unit normal pointing into $M$ and $H_{Eucl}$ is the mean curvature of $\dm$ 
with respect to the Euclidean metric; this is given by  
$$H_{Eucl} = \frac{N_E(r)}{r} + \k_E,$$
corresponding to the standard formula for the mean curvature of surfaces of revolution in $\bR^3$; here $\k_E$ is the Euclidean 
geodesic curvature of $\s$, given as $\<\nabla_T N, T\>$ in the Euclidean metric. 

\medskip 

{\bf Examples}.   A very useful and large class of examples of Weyl solutions, closely related to Newtonian gravity, are given 
as follows. Let $d\mu$ be any positive Radon measure, compactly supported on the axis $A$. Then the Newtonian potential 
of $\nu$, i.e.
\be \label{New}
\nu(x) = -\int_{A}\frac{1}{|x-y|} d\mu_y,
\ee
gives an axisymmetric harmonic function on $\bR^3 \setminus {\rm supp} \, d\mu$ and so generates an AF Weyl solution. For example, 
taking 
$$d\mu = \frac{1}{2}d\ell_{[-m,m]},$$
where $d\ell$ is the Lebesque measure on the interval $[-m,m]$ gives the Schwarzschild metric of mass $m$. On the other hand, 
taking 
$$d\mu = \k d\ell_{[-m,m]},$$
for any $\k \neq \frac{1}{2}$ generates a metric $g$ which is singular, i.e.~not smooth up to the boundary ${\rm supp}\, d\mu$. 

  Taking $d\mu = m\d_0$ to be a multiple of the Dirac delta function supported at the origin of $A$, gives 
$$\nu = -\frac{m}{R},$$
a multiple of the Green's function on $\bR^3$, with $R^2 = r^2 + z^2$. The resulting solution, called the Curzon solution, is given explicitly as 
$$g_C(m) = e^{2m/R}[e^{-m^2r^2/R^4}(dr^2 + dz^2) + r^2 d\f^2].$$
This metric becomes highly singular on approach to the origin $0 \in \bR^3$. 

  It is not difficult to see (cf.~\cite{A2}) that for all solutions as in \eqref{New}, $\nu \to -\infty$, so $u \to 0$, on approach to at least a 
dense subset of ${\rm supp} \, d\mu$. Thus, such static vacuum solutions are maximal, in the sense that they cannot be extended 
to any larger domain.\footnote{One could also take potentials $\nu$ as in \eqref{New} with more general signed measures or even 
distributions supported in $A$, generating Weyl metrics in the same way. However, when $d\mu$ is a positive measure, boundaries 
have fill-ins $(\O, g_{\O})$ with non-negative scalar curvature and the ADM mass $m_g$ of the solution is always non-negative.}  

   One may then take a boundary $\dm = {\rm Im}\, F$ as above to obtain a Weyl solution $(M, g, u)$ with induced Bartnik boundary data 
$(\g, H)$ on $\dm$. These solutions correspond to varying domains in a fixed ambient maximal Weyl solution generated by $\nu$ in \eqref{New}. 

\medskip 

  Not all Weyl solutions $(M, g, u)$ are of this form however. The most general solutions are given by solving the Dirichlet problem 
for $\nu$.\footnote{One could also solve the Neumann or a suitable Robin boundary value problem instead.} 
  
\begin{proposition}
Let $\s: I \to \bR^2$ be a $C^{m+1,\a}$ embedding as in \eqref{sig} and let $\nu$ be a $C^{m,\a}$ function on ${\rm Im}\, \s$. Then $\nu$ extends 
uniquely to an axisymmetric harmonic function on the exterior domain $M = \bR^3 \setminus B$ with $\nu \to 0$ at infinity, and generates an 
AF Weyl metric $(M, g, u)$ which is $C^{m,\a}$ up to $\dm$. 

  Conversely, up to isometry in ${\rm Diff}_1^{m+1,\a}(M)$, any AF Weyl metric $(M, g, u)$ which is $C^{m,\a}$ up to $\dm$ is uniquely 
given by such a pair $(\s, \nu)$ satisfying \eqref{norm}. 

\end{proposition}

\begin{proof} 
The first statement is standard from elliptic regularity theory, given the discussion above. For the second statement, given a specific representative 
$(M, g, u) \in \bE_{S^1}^{m,\a}$ representing a class in $\cE_{S^1}^{m,\a}$, construct the Weyl canonical coordinates $(r, z, \f)$, 
normalized as in \eqref{norm}. These coordinates restrict to $\dm$ to give an axi-symmetric embedding $F: \dm \simeq S^2 \to \bR^3$ 
and hence the embedding $\s: I \to (\bR^2)^+$. A different representative $(M', g', u')$ of $(M, g, u)$ in $\cE_{S^1}^{m,\a}$ gives the 
same map $F$, since the diffeomorphisms in ${\rm Diff}_1^{m+1,\a}(M)$ fix the boundary pointwise. 

  To see that $\s \in C^{m+1,\a}(\t)$, by hypothesis $\nu$, and hence $\l$, is $C^{m,\a}$ up to $\dm$, while $H \in C^{m-1,\a}(\t)$. 
Boundary regularity for harmonic functions implies that $N(\nu)$ is $C^{m-1,\a}$ up to $\dm$ and hence the same for $N(\l)$. Thus, 
by \eqref{Hform}, $H_{Eucl} \in C^{m-1,\a}(\t)$. This together with the fact that the Euclidean arclength parameter $\sqrt{(r')^2 + (z')^2} 
\in C^{m,\a}$ implies $\s \in C^{m+1,\a}(\t)$. 

\end{proof} 

\begin{remark}
{\rm Generic boundary data $\nu$ on $\dm$ lead to harmonic functions on $M$ which do not extend to any larger region $M' \supset M$ 
containing $\dm$ in the interior. Thus the examples above, although simple and useful, are not generic. Note that when 
the boundary data $\dm$ and $\nu$ are analytic, then $\nu$ and so $g$ does extend past $\dm$ to a slightly larger domain $M'$. 
}
\end{remark}

\begin{remark} 
{\rm Proposition 4.1 suggests that the simplest boundary data for a Weyl metric are the Dirichlet boundary data. This 
corresponds to the isometric embedding of a prescribed metric $u^2 dt^2 + \g$ on the boundary $S^1\times S^2$ into an ambient 
4-dimensional Weyl metric. However, it is proved in \cite{A3} that Dirichlet boundary data are not elliptic boundary 
data for (general) Einstein metrics. 
}
\end{remark} 

  We see then that the space $\bE_W^{m,\a}$ of $C^{m,\a}$ AF Weyl metrics in Weyl canonical coordinates has a very explicit 
(and simple) form:
$$\bE_W^{m,\a} \simeq [{\rm Emb}^{m+1,\a}(I) \times C^{m,\a}(I)] / T,$$
$$(M, g, u) \leftrightarrow (\s, \nu).$$
where $T$ is the action of translations along the $z$-axis $A$ on the first factor. The Bartnik boundary map is thus given by 
\be \label{pi2}
\Pi_{B}: \bE_W^{m,\a} \to \cB_{S^1}^{m,\a},
\ee
$$\Pi_{B}(\nu, \s) = (F^{*}g, H_{F, g}).$$ 
The three (arbitrary) functions $(\s, \nu) = (r(\t), z(\t), \nu(\t))$ describing the points in $\bE_W$ 
correspond to the three (arbitrary) functions $(\a(\t), \b(\t), H(\t))$ in $\cB_{S^1}$. 

Now, in contrast to the (abstract) map $\Pi_B$ in \eqref{Bart3}, the map in \eqref{pi2} is no longer smooth; although it is $C^0$, it is not 
even $C^1$. Namely, let $F_s$ be a smooth curve of $C^{m+1,\a}$ axi-symmetric embeddings of $S^2$ into $\bR^3$, for example 
$F_s = F + sX$ where $X$ is a $C^{m+1,\a}$ axi-symmetric vector field on $M$. Then the derivative $D\Pi_B(X)$ in \eqref{pi2} involves 
the terms $X(\l - \nu)$ and $X(\nu)$ (cf.~\eqref{gamma}) which are only $C^{m-1,\a}$ and not $C^{m,\a}$ on $M$ up to $\dm$.\footnote{This 
loss of derivative for the action of diffeomorphisms on the space of metrics is well-known. Note that $\Pi_B$ is smooth when $\nu$ is, 
say, $C^{\infty}$ up to $\dm$, so that $\Pi_B$ is smooth in the context of Frechet manifolds. If $(M, g, u)$ extends to a slightly larger 
domain $M'$, as in Remark 4.2, then $\nu$ is automatically as smooth as $\dm$ in $M'$, since static vacuum Einstein pairs $(g, u)$ 
are analytic in the interior.}  

Thus, the natural identification provided by Proposition 4.1 
$$\Phi: \cE_{S^1}^{m,\a} \simeq \bE_{W}^{m,\a},$$
is a homeomorphism, but is not a $C^1$ identification. In other words, for a $C^{\infty}$ curve of $C^{m+1,\a}$ diffeomorphisms 
$\f_s: M \to M$ with $\f_s \neq {\rm Id}$ on $\dm$, the curve $(\f_s^*g, \f_s^*u)$ is not in general a smooth curve in the abstract space 
$\cE_{S^1}$ or $\cE$. In the abstract setting, the ``location" of the boundary $\dm$ is not specified and the Bartnik boundary map 
should be viewed as a free boundary value problem. Nevertheless, the lack of smoothness of the identification $\Phi$ is not an 
essential issue for most purposes. 
 
\medskip 

   Despite the concrete setting of the Bartnik boundary map in \eqref{pi2}, it does not appear to be significantly easier to understand 
the validity of Conjectures I and II for Weyl metrics. For instance, we are not aware of any substantially simpler proofs of 
Proposition 2.1 or Theorem 3.5 in this setting. 

\medskip 

  To make some progress, let us consider a much simpler setting. Namely, fix the (global) potential $\nu$ as in \eqref{New}. Then 
$\nu$, $\l$ are globally defined, independent of the location of the boundary. Let 
\be \label{Enu}
\bE_{\nu}^{m,\a} \subset \bE_{W}^{m,\a}
\ee
be the subspace with potential fixed in this way. Clearly, $\bE_{\nu}$ is a Banach submanifold. For the restriction of the Bartnik boundary 
map to $\bE_{\nu}$, it is natural to drop the map to the mean curvature (freezing a scalar field in the domain corresponds to freezing a 
scalar field in the target).\footnote{Another natural choice would be to keep the mean curvature but take only the conformal class of the 
metric; cf.~\cite{A4} for further discussion.} This leads to consideration of the Dirichlet boundary data map 
\be \label{Dir}
\Pi_D: \bE_{\nu}^{m,\a} \to \hat {\rm Met}^{m,\a}_{S^1}(S^2),
\ee
$$\s \to F^{*}g.$$
Here $\hat {\rm Met}^{m,\a}_{S^1}(S^2)$ is the modification of ${\rm Met}^{m,\a}_{S^1}(S^2)$ consisting of $\a$, $\b$ as before with 
$\a \in C^{m,\a}(\t)$ but $\b \in C^{m+1,\a}(\t)$. This modification is due to the different levels of differentiability of $\a$ and $\b$ 
in \eqref{gamma}.\footnote{The regularizing property of the mean curvature is no longer present in this situation.} Also $g = g(\nu)$. 
Clearly \eqref{Dir} corresponds to the isometric embedding problem: given an axially symmetric metric on $S^2$, does there 
exist a (unique) isometric embedding into a Weyl metric $(M, g, u)$ with fixed $u$?  
  
    Consider first the case $u = 1$ so $\nu = 0$ and $g = g_{Eucl}$. Now the map $\Pi_D$ is not onto, i.e.~not every axi-symmetric metric on 
$S^1$ is realized by an isometric embedding or immersion as a surface of revolution in $\bR^3$. The following is a well-known 
necessary condition. 

\begin{lemma}
Let $\s(\t) = (r(\t), z(\t))$ be a $C^{m+1,\a}$ immersion into $(\bR^2)^+$ generating the surface of revolution $F$ with induced metric 
$\a^2 d\t^2 + \b^2 d\f^2$. Then for all $\t$,
\be \label{ab}
\a(\t) \geq |\b'(\t)|.
\ee
Further equality holds at the endpoints $\t = 0, \pi$, i.e.~at the poles of the surface. 
\end{lemma}

\begin{proof} For such an immersion, one clearly has 
\be \label{alpha}
\a^2 = (r')^2 + (z')^2
\ee
\be \label{beta}
\b = r,
\ee
 so that 
\be \label{z2}
(z')^2 = \a^2 - (\b')^2 \geq 0.
\ee
Since necessarily $\a > 0$, this proves \eqref{ab}. Further, as noted above, smoothness at the poles 
implies $z' = 0$, proving the second statement. 

\end{proof}
 
   Although it is obvious, it is worth pointing out explicitly that \eqref{beta} shows that the function $r(\t)$ of $\s(\t) = (r(\t), z(\t))$ is uniquely 
determined by the metric $\g$. Thus we need only consider the behavior of $z$ in terms of the boundary data $(\a, \b)$. 
 
   Let $\cV_0 \subset \hat{\rm Met}_{S^1}^{m,\a}(S^2)$ be the open set such that 
\be \label{ab2}
\a(\t) > |\b'(\t)|, \ \ \forall \t \in (0,\pi),
\ee
with $\b'(\t) = \a(\t)$ for $\f = 0,\pi$. Let $\cU_0 = \Pi_D^{-1}(\cV_0) \subset \bE_{\nu=0}^{m,\a}$. 

\begin{lemma} 
The map 
$$\Pi_D: \cU_0 \to \cV_0$$
is a diffeomorphism. 
\end{lemma} 

\begin{proof} By \eqref{z2}, one has 
\be \label{zp}
z' = \pm \sqrt{\a^2 - (\b')^2}.
\ee
Since by hypothesis, $z' \neq 0$ on $(0,\pi)$, there is a unique choice of sign for $z'$. As noted following \eqref{sig}, the choice of 
orientation of $\s$ then gives 
$$z' = -\sqrt{\a^2 - (\b')^2},$$
on $[0,\pi]$. By integration, this then uniquely defines $z$ under the normalization condition \eqref{norm}. As noted above, the 
function $r(\t)$ is already defined by $\b$. Clearly, the curve $\s$ is a graph over the $z$-axis (except at the poles), so that $\s$ 
is an embedding. 

  The remarks above prove that $\Pi_D$ is a bijection on $\cU_0$. It is also smooth, with derivative given by 
\be \label{dpid}
D\Pi_D(\rho, \z) = (\frac{1}{\a}(r' \rho + z' \z), \rho).
\ee
Since $z' \neq 0$, ${\rm Ker}D\Pi_D = 0$. Similarly, it is easy to see that $D\Pi_D$ is surjective, so that $\Pi_D$ is a diffeomorphism.  

\end{proof} 

  However $D_{\s}\Pi_D$ is no longer a Fredholm map at $\s$ if 
\be \label{zprime}
(\a - |\b'|)(\t_0) = 0,
\ee
i.e.~$z'(\t_0) = 0$, for some $\t_0 \in (0,\pi)$. Namely, to solve $D\Pi_D(\rho, \z) = (q, \rho)$ as in \eqref{dpid} requires $r'\rho + z'\z = \a q$, so that 
$$\z = \frac{\a q-r'\rho}{z'},$$
where $\a, \rho, r'$ and $z'$ are given.  This is solvable only if $\a q-r'\rho$ vanishes on the zero-set of $z'$, and the quotient $\z$ is then 
only $C^{m-1,\a}$. Thus $D\Pi_D$ maps at most only onto a dense subset of a codimension one subspace of $C^{m,\a}$ so that $\Pi_D$ 
is not Fredholm.\footnote{To avoid this loss of derivative, one could pass to $C^{\infty}$ maps and so Frechet spaces. While $D_{\s}\Pi_D$ 
would be Fredholm if $\s$ has isolated points where \eqref{zprime} holds to finite order, $\Pi_D$ would not be a tame 
Fredholm map, due to the factor $(z')^{-1}$.} 

   Next let $\cV_1 \subset \hat{\rm Met}_{S^1}(S^2)$ be the subset such that 
\be \label{chi}
\chi(t) = \a(\t) - |\b'(t)| \geq 0
\ee
with $\chi = 0$ at only finitely many points $\t_i \in [0,\pi]$, each of which is a non-degenerate minimum of $\chi$. As noted above, 
points where $\chi = 0$ correspond to points where $z' = 0$, i.e.~the tangent line to $\s$ is horizontal. As simple examples show, 
one can no longer expect that curves having points where $z' = 0$ remain embedded in general.\footnote{Compare with the discussion 
at the beginning of \S 3.} For this and related reasons, we modify (without changing the notation) the definition of $\bE_{\nu=0}$ to 
$$\bE_{\nu=0}^m = {\rm Imm_0}^{m+1}(I) / {\rm Isom}(\bR).$$
Thus, we have passed from oriented embeddings to unoriented immersions isotopic to embeddings relative to the endpoints, changed 
$(m, \a)$ to $m$ (since there are no longer any elliptic regularity issues) and taken the quotient by action of the rotation invariant 
subgroup ${\rm Isom}(\bR) \subset {\rm Isom}(\bR^3)$, consisting of translations along the $z$-axis $A$ and reflections in 
planes $z = const$. The group ${\rm Isom}(\bR)$ is the relevant group of Euclidean congruences. 

  As before, we have the Dirichlet (boundary) map 
$$\Pi_D: \bE_{\nu=0}^m \to \hat{\rm Met}_{S^1}^m(S^2), \ \ \s \to F^*(g_{Eucl}).$$
 As above, set $\cU_1 = (\Pi_D)^{-1}(\cV_1) \subset \bE_{\nu=0}^m$. It is worth noting that neither $\cV_1$ nor 
 $\cU_1$ is connected; each space has many components. For instance, passing from points in $\cU_0$ to points in $\cU_1$ 
 requires passing through curves where $\chi(t)$ has degenerate minima. 
 
 \begin{lemma} The smooth (but not Fredholm and so not proper) map 
 $$\Pi_D: \cU_1 \to \cV_1,$$
 is a bijection. Further ${\rm Ker}D\Pi_D = 0$ everywhere. 
 \end{lemma}
 
 \begin{proof} Let $(\a, \b) \in \cV_1^m$ be arbitrary. As with the proof of Lemma 4.5, the issue here is the determination of the 
 sign of $z'$ in terms of $\a, \b$. Once this is determined, $z$ and hence $\s = (r, z)$ is determined as above. 

  We first note that the order of vanishing of $z'$ is completely determined by $\a, \b$. Namely, by \eqref{z2}, 
\be \label{z'}
z' z'' = \a \a' - \b' \b'',
\ee
and 
\be \label{z''}
(z'')^2  + z'z'''  = (\a \a')' - (\b' \b'')'.
\ee
Thus at $\t_i$ where $z' = 0$, whether $z'' \neq 0$ or $z'' = 0$ is completely determined by whether the minimum of $\a - |\b'|$ is a 
non-degenerate minimum or not. 

  Now starting at the pole $\t = 0$, we have either $z'' > 0$ or $z'' < 0$. Suppose $z'' > 0$. Then the sign of $z'$ and hence 
as before $z$ and so $\s$, is uniquely determined up to first zero $\t_1$ of $\chi$. Since $z''(\t_1) \neq 0$, by continuity the 
sign of $z'$ is uniquely determined slightly past $\t_1$, so that $\s$ is then uniquely determined up to the second zero 
$\t_2$ of $\chi$. Proceeding in this way uniquely determines a curve $\s_1$ on $[0,\pi]$ with $\Pi_D(\s_1) = (\a, \b)$. 

  If $z''(0) < 0$, one may perform the same process, obtaining a curve $\s_2$ also with $\Pi_D(\s_2) = (\a, \b)$. However, up 
to translations along the $z$-axis, the $\bZ_2$ reflection through the plane $z = z(0)$ maps $\s_1$ to $\s_2$, so that 
$\s_1$ and $\s_2$ are congruent and represent the same point in $\cU_1$. This proves $\Pi_D$ is a bijection on $\cU_1$. 
The second statement follows easily from \eqref{dpid} as above. 

\end{proof} 

  One may proceed inductively in this way to higher order degeneracies of $z'$. Thus, let $\cV_2 \subset \hat{\rm Met}_{S^1}^m(S^2)$ 
be the subset such that the minima $\t_i \in (0,\pi)$ of $\chi$ are non-degenerate at order either $2$ or $4$; note that $\cV_1 \subset \cV_2$. 
As before, set $\cU_2 = (\Pi_D)^{-1}(\cV_2) \subset \bE_{\nu=0}^m$. 
Differentiating \eqref{z''} twice gives 
$$3(z''')^2 + 4z''z^{(4)} + z'z^{(5)}  = (\a \a')''' - (\b' \b'')''' = \tfrac{1}{2}((\a^2)^{(4)} - (\b^2)^{(5)}).$$
Thus, for $\s \in \cU_2$ (with $m \geq 4$) if $z'(\t_i) = z''(\t_i) = 0$ for some $\t_i$, then $z'''(\t_i) \neq 0$. 

  Note that smoothness requires $z''' = 0$ at the poles $0,\pi$. In any case, choose a generic point $\bar \t$ sufficiently near $0$ so 
that $(\a \a' - \b' \b'')(\bar \t) \neq 0$. By \eqref{z'}, $z'(\bar \t) \neq 0$ and $z''(\bar \t) \neq 0$. Suppose $z''(\bar \t) > 0$. Then the 
sign of $z'$ is uniquely determined by $\a, \b$ and so by \eqref{zp}, $z$ is uniquely determined by $\a, \b$ near $\bar \t$. The 
same arguments as above then show that $(\a, \b) \in \cV_2$ uniquely determine a curve $\s_1 \in \cU_2$. Again if $z''(\bar \t) 
< 0$, then the same construction produces another curve $\s_2 \in \cU_2$ which is congruent in ${\rm Isom}(\bR)$ to $\s_1$. 

  One may construct in the same way spaces $\cV_k \subset \cV_{k+1} \subset \cdots$, where $\cV_k$ consists of metrics for 
which $\chi$ has isolated zeros $\t_i$ each of which is non-degenerate at some order $\leq 2k$. This of course requires 
$m \geq 2k$. In the same way, given sufficient smoothness, one has the inclusion of $\Pi_D$-inverse images 
$\cU_k \subset \cU_{k+1} \cdots$. The same proof of Lemmas 4.5 - 4.6 as above shows that 
$$\Pi_D: \cU_k \to \cV_k,$$
is a bijection. Similarly, ${\rm Ker}D\Pi_D = 0$ on $\cU_k$. 

 Suppose however $\s \in \cU_k$, $m \geq 2k$ is a curve for which there exist zeros $\t_i$ and $\t_j$ of $\chi$ (not necessarily 
consecutive) at which $\chi$ vanishes at all orders $\leq 2(k-1)$ and is non-vanishing at order $2k$. Then the derivatives $z^{(j)} = 0$ for all 
$j \leq k$ at $\t_i$ and $\t_j$ but $z^{(k+1)} \neq 0$ at $\t_i$ and $\t_j$. Suppose also  
$$z(\t_i) = z(\t_j).$$
One may perform a $\bZ_2$-reflection of $\s$ through the plane $z = z(\t_i)$ over the interval $[\t_i, \t_j]$ to obtain another curve $\w \s$ 
which agrees with $\s$ outside $[\t_i, \t_j]$. These two curves differ only by a local, not global, $\bZ_2$-reflection and so are 
non-congruent curves. Note however that while $\s \in \bE_{\nu=0}^{m}$, with $m \geq 2k$, $\w \s$ is only in $\bE_{\nu=0}^{k}$, so 
that these curves are not in the same space. This loss of derivatives can be cured (only) by going to $C^{\infty}$, and when $z'$ vanishes 
to infinite order at $\t_i$ and $\t_j$. Thus there are regions where $\Pi_D$ becomes (at least) a $2-1$ map in certain regions in the $C^{\infty}$ 
context. 

  Finally, to conclude this discussion suppose that $\chi = \a - |\b'|$ on an open $\t$-interval $(\t_0, \t_1)$. The surface $\Si = {\rm Im}F$ 
is then a flat annulus in this region. Let $f$ be any smooth function of compact support in $(\t_0, \t_1)$. Then the deformation 
$$\s_s = \s + sfN,$$
$N = \p_z$, induces an infinitesimal isometric deformation of $\Si$, $\cL_{fN}\g = 0$, Thus ${\rm Ker}D\Pi_D$ is now 
infinite dimensional. Also, for any $s$ and $f$ as above, the surfaces $\Si_{\pm}$ generated by
$$\s_{\pm} = \s \pm sfN$$
are isometric but not congruent. Thus, one has here the $2-1$ fold behavior discussed in \S 2 and \S 3. 

  Summarizing, the discussion above proves: 
  
\begin{proposition}
For any $\g = (\a, \b) \in {\rm Met}_{S^1}^{\infty}(S^2)$ satisfying \eqref{ab}, there is a $C^{\infty}$ axi-symmetric isometric immersion 
$F:S^2 \to \bR^3$, so that 
$$F^*(g_{Eucl}) = \g.$$
The same result holds for $C^m$ smoothness, for any $1 \leq m < \infty$. However, in general the map 
$$\Pi_D: {\rm Imm}_{S^1}^{\infty}(S^2) \to {\rm Met}_{S^1}^{\infty}(S^2),$$
may be finite-to-one or $\infty$-to-one. 
\end{proposition}

\begin{remark}
{\rm All of the results above hold for general $\bE_{\nu}$ as in \eqref{Enu}, when the condition \eqref{ab} is replaced by 
$$e^{\l}\a \geq |\b'|,$$
where $\l$ is determined from $\nu$ as in \eqref{lam}. 
}
\end{remark}
    
    We also note that the infinitesimal rigidity (i.e.~${\rm Ker}D\Pi_D  = 0$) on the spaces $\cU_k$ discussed above fails badly when 
one considers non-axisymmetric embeddings of axi-symmetric metrics $S^2 \to \bR^3$.\footnote{Thus the symmetry of the metric is not 
assumed apriori to extend to a symmetry of $\bR^3$.} There is a large classical literature on this, starting from the remarkable 
examples of Cohn-Vossen \cite{CV}, see also \cite{Rm}, \cite{Sp}, \cite{Sab}.  
 
   In conclusion, it would be interesting to understand if some of the complicated behavior of the isometric embedding problem
for surfaces of revolution in $\bR^3$ exhibited above carries over to general Weyl metrics and the corresponding Bartnik boundary 
map in \eqref{Bart3} or more generally in \eqref{bmap}.

\section{Appendix}

  In this Appendix, we first collect several standard facts regarding static vacuum Einstein metrics. Following this, we prove a regularity result for 
minimal surfaces in static vacuum spaces used in the proof of Theorem 3.5. 

\medskip 

On a Riemannian manifold $(M, g)$ with (local) boundary $(\dm, \g)$ the Gauss and Gauss-Codazzi equations (also known as the Hamiltonian and 
momentum constraint equations) are:
\be \label{Hamcon}
|A|^2 - H^2 + s_{\g} = s_g - 2{\rm Ric}(N,N),
\ee
\be \label{divcon}
\d A + dH = -{\rm Ric}(N, \cdot),
\ee
where $N$ is the unit normal to $\dm$ and $\d$ is the divergence operator on $\dm$. For static vacuum Einstein metrics, $s_g = 0$ and 
${\rm Ric}(N,N) = u^{-1}NN(u)$. Using the standard relation $\D_g v = NN(v) + HN(v) + \D_{\dm}v$, \eqref{Hamcon} gives 
\be \label{Hamcon2}
2u^{-1}(\D_{\dm}u + HN(u)) = |A|^2 - H^2 + s_{\g}.
\ee
Similarly, the static vacuum equations give ${\rm Ric}(N, \cdot) = u^{-1}D^2u(N, \cdot) = u^{-1}(dN(u) - A(du))$. It follows that 
\be \label{dnu}
\d (uA) + udH = - dN(u).
\ee
or equivalently 
$$\d(uA - N(u)\g) = -udH.$$

\medskip 

  Next, consider the conformally related metric $\w g = u^2 g$. Standard formulas for conformal change of metric show that the static vacuum Einstein 
equations \eqref{stat} are equivalent to the system
\be \label{tilde}
{\rm Ric}_{\w g} = 2 d\nu \cdot d\nu, \ \ \D_{\w g} \nu = 0,
\ee
where $\nu = \log u$. Two advantages of these equations are that the Ricci curvature ${\rm Ric}_{\w g}$ is positive, and second, the Ricci curvature is lower 
order (1st order) in $\nu$, which is not the case for \eqref{stat}. 

\medskip 

  Next we turn to the regularity properties of local minimal surfaces $(V, \g)$ in static vacuum Einstein metrics $(M, g, u)$. To begin, as in \eqref{Kbound}, 
assume a bound on $Q$ in \eqref{Q}, so 
\be \label{Q2}
|s_{\g}| + |{\rm Rm}_g| \leq C,
\ee
on $(V, \g)$. 

  We first discuss what uniform control the bound \eqref{Q2} gives on the geometry of $(M, g)$ and $(V, \g)$. 
It follows from the Cheeger-Gromov compactness theorem (or more simply from the uniformization theorem in 2 dimensions) that 
$\g$ is controlled, in harmonic coordinates, in $C^{1,\a}$. Similarly, $g$ is also controlled in $C^{1,\a}$, cf.~also \cite{AT}. (In case the 
geometry is (arbitrarily) collapsed, one may unwrap the collapse in universal covers and obtain the same conclusion). 

  Since $H = 0$, it follows first from the Hamiltonian constraint \eqref{Hamcon} that $A$ is uniformly bounded in $L^{\infty}$, 
$$|A| \leq C, \ \ {\rm on} \ \ V.$$  
Next, assume (as in \S 3) that $u$ is normalized at a base point $p \in V$ so that $u(p) = 1$. It then follows from \eqref{Hamcon2} and elliptic regularity 
for the Laplacian on $(V, \g)$ that $u$ is uniformly bounded in $L_{loc}^{2,p}$, for any $p < \infty$: 
\be \label{up1}
|u|_{L_{loc}^{2,p}(V)} \leq C.
\ee
Since $\D_g u = 0$ on $(M, g)$, elliptic regularity for the Laplacian on $(M, g)$ implies that $u$ is $L^{2,p}$ and so $C^{1,\a}$ up to 
$\dm$: 
\be \label{up2}
|u|_{L_{loc}^{2,p}(M)} \leq C,
\ee 
locally near $V$. Hence $N(u)$ is bounded in $L_{loc}^{1,p}$ and so in $C^{\a}$ on $V$. Finally, it then follows from the momentum constraint \eqref{dnu} 
and elliptic regularity for the $(\d, tr)$ elliptic system on $(V, \g)$ that $A$ is bounded in $L_{loc}^{1,p}$ and so $C_{loc}^{\a}$ on $V$: 
\be \label{Ac}
|A|_{C_{loc}^{\a}(V)} \leq C.
\ee 

  The main step in proving regularity is the following result. Recall $\w g = u^2 g$. 

\begin{lemma}  
The system 
\be \label{tildesys}
{\rm Ric}_{\w g} = 2 d\nu \cdot d\nu , \ \ \D_{\w g} \nu = 0,
\ee
on $(M, \w g, \nu)$ with boundary conditions 
\be \label{wH}
(\w \g, H)
\ee
where $H = u\w H - 2N(\nu)$ on $\dm$, is an elliptic boundary value problem in Bianchi gauge. 
\end{lemma}

We note that $H$ in \eqref{wH} is the mean curvature of $\dm$ with respect to $g$.  

\begin{proof}
  It suffices to prove the result for the linearized system and for deformations $h$ of $g_{\cM}$ satisfying the Bianchi gauge condtion 
$\b (h) = \d h + \frac{1}{2}d \tr h = 0$, where the divergence and trace are with respect to the 4-metric $g_{\cM}$ as in \eqref{st}. 
In the interior, the leading order term is just ${\rm Ric}'(h) \sim -\frac{1}{2}\D _{\w g} h$, which is elliptic with principal symbol $-|\xi|^2 I$. 
The proof that the boundary conditions \eqref{wH} are elliptic is exactly the same as that given in \cite{A3}, \cite{AK} for the boundary 
conditions $(\w \g, \w H)$; thus we refer to \cite{A3} or \cite{AK} for details.

\end{proof} 

\begin{proposition} 
Suppose $(V, \g)$ is a (local) minimal surface in a static vacuum Einstein manifold $(M, g, u)$ satisfying the bound \eqref{Q2}. Then the geometry of 
$(V, \g)$ is bounded in the $C^{\infty}$ norm. 
\end{proposition} 

\begin{proof} 
The proof follows the usual bootstrap method in elliptic PDE. To begin, by \eqref{up2} we have uniform $C^{1,\a}$ control on $\nu$. By \eqref{tildesys} this gives 
uniform $C^{\a}$ control on $\w {\rm Ric}$, and hence uniform $C^{2,\a}$ control on $\w g$ in the interior of $M$. Next, the 
Hamiltonian constraint \eqref{Hamcon} in the $\w g$ metric gives 
$$|\w A|^2 - \w H^2 + s_{\w \g} = s_{\w g} - 2{\rm Ric}_{\w g}(\w N, \w N).$$
The right side is 1st order in $\nu$, and $\w A$, $\w H$ are uniformly controlled in $C^{\a}$ by \eqref{Ac} and the control on $\nu$. Hence $s_{\w \g}$ 
is uniformly controlled in $C^{\a}$. As above, by the uniformization theorem (or Cheeger-Gromov compacness theorem) it follows that $\w \g$ is uniformly 
controlled in $C^{2,\a}$ modulo diffeomorphisms, i.e.~in harmonic coordinates. It then follows from boundary regularity for elliptic systems, cf.~\cite{AT} 
that $\w g$ is uniformly controlled in $C^{2,\a}$ up to the boundary $V$. This in turn gives $C^{1,\a}$ control on $\w H = \cL_{\w N}dv_{\w g}$ and so, 
since $H = 0$, $C^{1,\a}$ control on $N(\nu)$. Via the elliptic equation $\w \D \nu = 0$, this gives uniform $C^{2,\a}$ control on $\nu$ up to the 
boundary $V$ and hence uniform control of $g$ in $C^{2,\a}$ up to the boundary $V$. This then also gives uniform control of $A$ in $C^{1,\a}$ on $V$. 

   This shows that the initial regularity of $(M, g)$ up to the boundary $(V, \g)$ has been increased by one derivative. Iterating this process then 
proves the result. 

\end{proof}

\bibliographystyle{plain}

\end{document}